\documentclass[a4paper, 11pt, reqno]{amsart}
\usepackage[top = 1in, bottom = 1in, left=1.25in, right=1.25in]{geometry}
\usepackage{setspace} 

\usepackage[utf8]{inputenc}
\usepackage[T1]{fontenc}
\usepackage{lineno} 
\newcounter{propcounter}
\usepackage[inline]{enumitem}
\setlist[enumerate,1]{label={\upshape (\roman*)}}

\usepackage{amssymb,amsmath,amsthm,graphicx,xcolor,mathtools,mathrsfs}
\usepackage{comment}

\usepackage{hyperref}
\hypersetup{colorlinks, linkcolor={red!50!black}, citecolor={green!50!black}, urlcolor={blue!50!black}}

\usepackage[
backend=biber,
natbib=true,
style=numeric,
isbn=false,
eprint=false,
url=false,
doi=false,
maxbibnames=9
]{biblatex}
\usepackage{csquotes}
\renewbibmacro{in:}{}
\DeclareFieldFormat{pages}{#1}
\DeclareFieldFormat[article,misc]{volume}{\mkbibbold{#1}\nopunct}
\DeclareFieldFormat[article,misc]{number}{(#1)}

\DeclareFieldFormat{title}{\mkbibquote{#1}}
\DeclareFieldFormat[misc]{date}{(#1)}
\DeclareFieldFormat[misc]{note}{#1\nopunct}

\renewbibmacro*{doi+eprint+url}{
	\iftoggle{bbx:url}
	{\iffieldundef{doi}{\usebibmacro{url+urldate}}{}}{}
	\newunit\newblock
	\iftoggle{bbx:eprint}
	{\usebibmacro{eprint}}
	{}
	\newunit\newblock
	\iftoggle{bbx:doi}
	{\printfield{doi}}
	{}
}

\newbibmacro{string+doi}[1]{%
	\iffieldundef{doi}{\iffieldundef{url}{#1}{\href{\thefield{url}}{#1}}}{\href{http://dx.doi.org/\thefield{doi}}{#1}}}	
\DeclareFieldFormat*{title}{\usebibmacro{string+doi}{#1}}

\usepackage[nameinlink]{cleveref}
\theoremstyle{plain}
\newtheorem{theorem}{Theorem}[section]
\crefname{theorem}{Theorem}{Theorems}

\newtheorem{proposition}[theorem]{Proposition}
\crefname{proposition}{Proposition}{Propositions}

\newtheorem{corollary}[theorem]{Corollary}
\crefname{corollary}{Corollary}{Corollaries}

\newtheorem{lemma}[theorem]{Lemma}
\crefname{lemma}{Lemma}{Lemmas}

\crefname{conjecture}{Conjecture}{Conjectures}

\crefname{problem}{Problem}{Problem}

\crefname{claim}{Claim}{Claims}

\newtheorem{observation}[theorem]{Observation}
\crefname{observation}{Observation}{Observations}

\crefname{setup}{Setup}{Setups}

\crefname{myth}{Myth}{Myths}

\crefname{fact}{Fact}{Facts}

\crefname{algorithm}{Algorithm}{Algorithms}

\crefname{remark}{Remark}{Remarks}

\crefname{example}{Example}{Examples}

\theoremstyle{definition}
\newtheorem{definition}[theorem]{Definition}
\crefname{definition}{Definition}{Definitions}

\crefname{construction}{Construction}{Constructions}

\crefname{question}{Question}{Questions}

\numberwithin{equation}{section}

\crefformat{section}{\S#2#1#3}
\crefformat{subsection}{\S#2#1#3}
\crefformat{subsubsection}{\S#2#1#3}
\crefrangeformat{section}{\S\S#3#1#4 to~#5#2#6}
\crefmultiformat{section}{\S\S#2#1#3}{ and~#2#1#3}{, #2#1#3}{ and~#2#1#3}

\DeclareMathOperator{\probability}{Pr}
\DeclareMathOperator{\expectation}{\mathbf{E}}

\newcommand{\eps}{\varepsilon}
\newcommand{\deltaLHC}[2]{\smash{\delta^{\mathrm{LHC}}_{#1, #2}}}
\newcommand{\deltaEST}[2]{\smash{\delta^{\mathrm{EST}}_{#1, #2}}}
\allowdisplaybreaks

\author[M. Rao]{Mengjiao Rao}
\address{Data Science Institute, Shandong University, Jinan,  China.}
\email{mengjiao\_rao@outlook.com}
\author[N.~Sanhueza-Matamala]{Nicol\'{a}s Sanhueza-Matamala}
\address{Departamento de Ingenier\'{\i}a Matem\'{a}tica, Facultad de Ciencias F\'{\i}sicas y Matem\'{a}ticas, Universidad de Concepci\'{o}n, Concepci\'{o}n, Chile.}  \email{nsanhuezam@udec.cl}
\author[L. Sun]{Lin Sun}
\address{School of Mathematics and Statistics, Qingdao University, Qingdao, China.}
\email{sun.lin@qdu.edu.cn}
\author[G. Wang]{Guanghui Wang}
\author[W. Zhou]{Wenling Zhou}
\address{School of Mathematics, Shandong University, Jinan,  China.}
\email{\{\,ghwang\,|\,gracezhou\,\}@sdu.edu.cn}

\title{Degree conditions for spanning expansion hypertrees}

\begin{document}

\begin{abstract}
The $k$-expansion of a graph $G$ is the $k$-uniform hypergraph obtained from $G$ by adding $k-2$ new vertices to every edge.
We determine, for all $k > d \geq 1$, asymptotically optimal $d$-degree conditions that ensure the existence of all spanning $k$-expansions of bounded-degree trees, in terms of the corresponding conditions for loose Hamilton cycles.
This refutes a conjecture by Pehova and Petrova, who conjectured that a lower threshold should have sufficed.
The reason why the answer is off from the conjectured value is an unexpected `parity obstruction': all spanning $k$-expansions of trees with only odd degree vertices require larger degree conditions to embed.
We also show that if the tree has at least one even-degree vertex, the codegree conditions for embedding its $k$-expansion become substantially smaller.
\end{abstract}

\maketitle
\thispagestyle{empty}

\vspace{-2.15em}

\section{Introduction}
Koml\'{o}s, Sark\"{o}zy and Szemer\'{e}di~\cite{Komlos-Sarkozy-Szemeredi-1995} proved that graphs of sufficiently large minimum degree contain all bounded-degree spanning trees as subgraphs.
We wish to generalize this fact to hypergraphs.
A \emph{$k$-uniform $\ell$-tree} with $m$ edges is a $k$-graph admitting an edge ordering $e_1, \ldots, e_m$ such that each $e_i$ shares exactly $\ell$ vertices with one previous edge in the ordering, and $k-\ell$ new vertices. 
We call $k$-uniform $\ell$-trees \emph{loose trees} if $\ell = 1$.
Given a $k$-graph $H$ and a subset $S\subseteq V(H)$, let $\deg_{H}(S)$ denote the number of edges containing $S$ in $H$.
The \emph{minimum $d$-degree} $\delta_{d}(H)$ is the minimum of $\deg_{H}(S)$ over all subsets $S\subseteq V(H)$ of size $d$ and define \emph{maximum $d$-degree} $\Delta_d(H)$ analogously.
We refer to the cases $d=k-1$ and $d=1$ as \emph{codegree} and \emph{vertex degree}, respectively.
 
For $k\ge 3$, Pavez-Sign\'{e}, Sanhueza-Matamala and Stein~\cite{Pavez-Matias-Nicolas-Stein-24} showed that any $n$-vertex $k$-graph with minimum codegree at least $(1/2 + o(1))n$ contains all spanning $(k-1)$-trees of bounded maximum vertex degree.
The constant $1/2$ here is best-possible, since it is necessary for the existence of tight Hamilton paths and cycles.
Pehova and Petrova~\cite{Pehova-Petrova-24} showed that for any $n$-vertex $3$-graph $H$, $\delta_1(H) \geq (5/9 + o(1))\binom{n}{2}$ suffices to embed all loose trees of bounded maximum vertex degree.
Here, the optimality of $5/9$ is not due to Hamilton paths but rather because loose trees can contain perfect matchings (see~\cite[Fig. 1]{Pehova-Petrova-24}), and we need such a minimum vertex degree to ensure their existence~\cite{Han-Person-Schacht-09}.

We thus study loose trees that do not contain perfect matchings.
We consider \emph{$k$-expansion hypertrees}, as proposed by Stein~\cite[\S 9.2]{Stein-2020} as well as~\cite[\S 6]{Pehova-Petrova-24}.
Given a graph $G$, its \emph{$k$-expansion} $G^{(k)}$ is the $k$-graph obtained from $G$ by replacing each edge $uv$ with a $k$-edge containing $u, v$ and $k-2$ new vertices that only participate in this edge.
The $k$-expansion of a tree $T$, denoted $T^{(k)}$, is a loose tree; and we say that $T^{(k)}$ is a \emph{$k$-expansion hypertree}.
Note that a $k$-expansion hypertree cannot contain perfect matchings for $k \geq 3$.

\subsection{Embedding expansion hypertrees}
Our main result establishes an optimal $d$-degree condition that ensures the existence of all spanning $k$-expansion hypertrees of bounded maximum vertex degree, based on the existence of loose Hamilton cycles (i.e., spanning $k$-expansions of cycles in graphs).

\begin{theorem} \label{theorem:main}
For all $k > d \geq 1$, $\eps>0$ and $\Delta\in \mathbb N$, there exists an $n_0\in \mathbb N$ such that any $k$-graph $H$ on $n\ge n_0$ vertices with \[\delta_d(H) \geq \left(\max\left\{ \frac{1}{2}, \deltaLHC{k}{d} \right\} + \eps \right) \binom{n-d}{k-d}\] contains $T^{(k)}$ for every tree $T$ with $\Delta_1(T)\le\Delta$ and $|V(T^{(k)})| = n$.
\end{theorem}

The $\deltaLHC{k}{d}$ in the statement is the \emph{asymptotic minimum $d$-degree threshold for loose Hamilton cycles in $k$-graphs} (\cref{def:LHC-threshold}).
We can also define a corresponding threshold $\deltaEST{k}{d}$ for spanning $k$-expansion hypertrees of bounded maximum vertex degree (\cref{def:EST-threshold}).
In a nutshell, \cref{theorem:main} says that $\deltaEST{k}{d} \leq \max\{\frac{1}{2}, \deltaLHC{k}{d} \}$ for all $k > d \geq 1$.
As we shall explain now, in fact the equality $\deltaEST{k}{d} = \max\{\frac{1}{2}, \deltaLHC{k}{d} \}$ holds in all cases.

Pehova and Petrova conjectured~\cite[Conjecture 6.2]{Pehova-Petrova-24} that the degree conditions ensuring loose Hamilton cycles should also give every spanning $k$-expansion hypertree of bounded maximum vertex degree, i.e., they conjectured that $\deltaEST{k}{d} \leq \deltaLHC{k}{d}$ for all $k$ and $d$.
Note that a loose Hamilton path (i.e., a spanning $k$-expansion of a path) is a $k$-expansion hypertree of bounded maximum vertex degree, so the degree condition has to be enough to ensure the existence of such loose paths.
The connection with $\deltaLHC{k}{d}$ arises because the asymptotic degree conditions ensuring the existence of loose Hamilton paths coincide with the asymptotic degree conditions that ensure loose Hamilton cycles (see \cref{section:conclusion-loosecycle} for discussion).
Hence, $\deltaEST{k}{d} \geq \deltaLHC{k}{d}$ holds.
Considering this, the conjecture of Pehova and Petrova essentially says that the loose Hamilton path should be the hardest $k$-expansion hypertree to embed.

We recall the known results on loose Hamilton cycles.
For all $k > d \geq 1$, we know
\begin{equation}
    \deltaLHC{k}{d} \geq 1 - \left(1 - \frac{1}{2(k-1)} \right)^{k-d}.
    \label{equation:lowerbound-deltaLHC}
\end{equation}
This follows from a `space barrier' construction (see \cite[Proposition 1.9]{Gan-Han-Sun-Wang-2023}).
Equality is known to hold for $d = k-1$ and all $k \geq 2$~\cite{Han-Schacht-10,Keevash-Kuhn-Mycroft-Osthus-11, Kuhn-Osthus-06}, and for $d = k-2$ and all $k \geq 3$~\cite{Bastos-Mota-Schacht-Schnitzer-Schulenburg-2017,Bub-Han-Schacht-13}.
Note that in all the solved cases, we have $\deltaLHC{k}{d} < \frac{1}{2}$ whenever $k \geq 3$.

However, we find a new parity obstacle for the embedding of $k$-expansion hypertrees.
Indeed, we construct a family of $k$-graphs and $k$-expansion hypertrees showing that $\deltaEST{k}{d} \geq 1/2$ for all $k > d \geq 1$ (\cref{theorem:construct}).
Together with the results on $\deltaLHC{k}{d}$ mentioned above, this refutes the conjecture of Pehova and Petrova in general, and shows that \cref{theorem:main} is best-possible.

\subsection{Parity obstructions and trees with even-degree vertices}
The $k$-expansion hypertrees used to show that $\deltaEST{k}{d} \geq 1/2$ are precisely bounded-degree trees all of whose vertices have odd degree (see \Cref{theorem:construct}).
In the opposite case (if $T$ has at least one vertex of even degree), this obstacle disappears.
We wondered if, given the presence of one even-degree vertex in $T$, maybe the degree condition of the loose Hamilton cycle is indeed sufficient to embed the $k$-expansion of those trees.
But this is also not the case: in fact we generalize our parity obstruction and show that for every $1 \leq d < k$, a degree condition of at least $\delta_d(H) \geq (1/q + o(1)) \binom{n-d}{k-d}$ is required to embed all spanning $k$-expansion hypertrees $\smash{T^{(k)}}$ whenever $T$ has only vertices of degree $1 \bmod q$ (\Cref{theorem:construct-q}).

In particular, we still need at least $\delta_d(H) \geq (1/3 + o(1)) \binom{n-d}{k-d}$ to embed spanning $k$-expansion hypertrees, even if we insist that the original tree has vertices of even degree.
Our second main result is that this is best possible for codegree conditions.

\begin{theorem} \label{theorem:main-even}
For all $k \geq 3$, $\eps>0$ and $\Delta\in \mathbb N$, there exists $n_0\in \mathbb N$ such that any $k$-graph $H$ on $n\ge n_0$ vertices with $\delta_{k-1}(H) \geq (1/3 + \eps) n$ contains $T^{(k)}$ for every tree $T$ with $\Delta_1(T)\le\Delta$ and $|V(T^{(k)})| = n$, as long as $T$ has at least one vertex of even degree.
\end{theorem}

We remark that, from the previously mentioned results, we have $\deltaLHC{k}{k-1} = 1/(2k-2)$, which is strictly less than $1/3$ for all $k \geq 3$.
Then our result shows that the loose Hamilton cycle threshold is not enough to embed all $k$-expansion hypertrees, even if we remove the parity obstacle that introduces the term $1/2$ in \cref{theorem:main}.

\subsection{Other results}

From \cref{theorem:main-even} we can obtain, as a simple corollary, that a codegree above $n/3$ suffices to embed bounded-degree expansion hypertrees as long as we leave some constant number of vertices (depending on $k$) uncovered.

\begin{corollary} \label{corollary:almost}
For all $k \geq 3$, $\eps>0$ and $\Delta\in \mathbb N$, there exists $n_0\in \mathbb N$ such that any $k$-graph $H$ on $n\ge n_0$ vertices with $\delta_{k-1}(H) \geq (1/3 + \eps) n$ contains the $k$-expansion of any tree $T$ with $\Delta_1(T)\le\Delta $, and such that $|V(T^{(k)})| \leq n - k + 1$.
\end{corollary}

Finally, we can get a stronger version of \Cref{theorem:main} for $d=k-1$, i.e., for codegree conditions, we can replace the bounded-degree assumption on $T$ for $\Delta_1(T) \leq c n / \log n$.

\begin{theorem} \label{theorem:nlogn}
For all $k \geq 2$, $\eps>0$, there exist $c > 0$ and $n_0\in \mathbb N$ such that any $k$-graph $H$ on $n\ge n_0$ vertices with $\delta_{k-1}(H) \geq ( 1/2 + \eps )n$ contains $T^{(k)}$ for every tree $T$ with $\Delta_1(T) \leq c n / \log n$ and $|V(T^{(k)})| = n$.
\end{theorem}

\Cref{theorem:nlogn} will be a quick corollary of a result on transversal embeddings of spanning trees by Montgomery, M\"{u}yesser and Pehova~\cite{Montgomery-Muyesser-Pehova-2022}.

\subsection{Ideas behind the proofs} \label{ssection:intro-sketch}
Now we will describe the main ideas behind our proofs of \Cref{theorem:main,theorem:main-even}.
In both proofs, we want to embed a spanning $k$-expansion hypertree $T^{(k)}$ into $H$.
For both of them we will use an `absorption' technique.
Very roughly, there are two main parts to our argument.
First, we embed an `almost spanning' subtree $T_{\mathrm{main}}^{(k)} \subseteq T^{(k)}$.
Secondly, we will find in $H$ a set of gadgets (the `absorbers') that allow us to transform the embedding of $T_{\mathrm{main}}^{(k)}$ into one of $T^{(k)}$.
To perform the first part, we use the `blow-up covers' method of Lang and the second author~\cite{Lang-Sanhueza-2024}, thus avoiding the need for any kind of Hypergraph Regularity Lemma.
This decreases the length and technicality of many of our arguments.

To implement the second part of the argument, we will implement the absorption via switchings.
Switchings have been a common mechanism behind absorption for different types of bounded-degree spanning structures in graphs and hypergraphs, as done by B\"ottcher et al.~\cite{Bottcher-Han-Kohayakawa-Montgomery-Parczyk-Person-2019, Bottcher-Montgomery-Parczyk-Person-2020}, and then applied successfully to the setting of hypertrees~\cite{Chen-Lo-2025, Pavez-Matias-Nicolas-Stein-24, Pehova-Petrova-24}.
In such a setting, we would have an embedding $\phi$ of a $k$-expansion hypertree $\smash{T^{(k)}_0}$ in a host $k$-graph $H$, such that $\smash{T^{(k)}_0}$ is a subgraph of our target $k$-expansion hypertree $T^{(k)}$.
We would like to extend the embedding $\phi$ by `adding a leaf'.
This means that for some $e \in E(T^{(k)}) \setminus E(T^{(k)}_0)$ with $e \cap V(T^{(k)}_0) = \{x\}$, we find an image for $e$ that is consistent with $\phi(x) = w$, which is already embedded.
To do this, we will need to use $k-1$ new unused vertices, say $\{w_1, \dotsc, w_{k-1}\}$.
Our gadgets will find a structure for the tuple $(w, w_1, \dotsc, w_{k-1})$ that modifies the embedding $\phi$ in a constant number of vertices and constructs an embedding for $T^{(k)}_0 \cup e$, thus extending our embedding by one edge.
The mechanism behind our gadgets is explained with more detail in \cref{ssection:sketch-gadgets}, and the strategy to prove \cref{theorem:main} is described with more detail in \cref{ssection:sketch-main}.

The strategy for \cref{theorem:main-even} is similar on the surface, but the construction of absorbing gadgets is more delicate and is the main technical obstacle to overcome in our proofs.
The obstacle appears because the $k$-graph $H$ can have a rigid partite structure, for instance if $V(H)$ has a partition $\{A,B\}$ where all edges have odd intersection with $A$
(this corresponds to $k$-graphs that are subgraphs of the $k$-graphs used in our lower bound in \Cref{theorem:construct}).
In this setting, in fact we cannot build absorbers for arbitrary tuples $(w, w_1, \dotsc, w_{k-1})$ as before.
This will be the case, for instance, if $w \in B$ and the set $\{w, w_1, \dotsc, w_{k-1}\}$ has even intersection with $A$, because then every edge containing $w$, by assumption, will have an odd intersection with $A$.
To perform one absorption step, then we are forced to select sets $\{ w_1, \dotsc, w_{k-1}\}$ in the leftover whose intersection with $A$ is odd.
This could be a problem if, for instance, the leftover has zero vertices in $A$.

To circumvent this issue, we will need to build `balancer gadgets' allowing us to move the vertices in an embedding between the parts $A$ and $B$.
Before proceeding with the absorbing steps, we will perform a `balancing step', where we will ensure the unused vertices in the leftover are roughly balanced between the vertices in $A$ and $B$.
This will ensure that, in every absorbing step, we can select a set $\{w_1, \dotsc, w_{k-1}\}$ with the correct parity constraints, so there is always an available absorber.
The strategy for proving \cref{theorem:main-even} is described with more detail in \cref{ssection:sketch-even}.

\subsection{Organization of the paper}
In \Cref{section:preliminaries} we fix notation and give more detailed sketches of the proofs of our two main results, \Cref{theorem:main} and \Cref{theorem:main-even}.
In \Cref{section:parity} we exhibit parity-based constructions showing that the degree conditions in our main results are best-possible.
In the next three sections we prove \Cref{theorem:main}: in \Cref{section:tools} we gather some facts and known results, in \Cref{section:almost} we show how to embed an almost-spanning $k$-expansion hypertree, and we finish the proof in \Cref{section:proof-main}.
In the following sections we prove \Cref{theorem:main-even}: in \Cref{section:basicgadgets}--\Cref{section:complexgadgets} we describe increasingly complex gadgets that we use in our proofs; in \Cref{section:immersion} we show how to embed a small $k$-expansion hypertree that covers the gadgets correctly, and we finish in \Cref{section:proof-even}.
In \Cref{section:short} we prove \Cref{corollary:almost} and \Cref{theorem:nlogn}, and we finish with concluding remarks in \Cref{section:conclusion}.

\section{Preliminaries} \label{section:preliminaries}
\subsection{Basic notions and notation}
Given a set $V$ and a positive integer $\ell$, we write $\binom{V}{\ell}$ to denote the family consisting of subsets of $V$ of size $\ell$. 
Given $k\ge 2$, let $H$ be a $k$-graph. We write $v_1\dots v_k$ for an edge $\{v_1, \dots, v_k\}$ of $H$.
For $S\subseteq V(H)$, the set of \emph{neighbors} of $S$ in $H$ is defined as
$N_H(S)=\{S'\subseteq V(H)\setminus S: S\cup S'\in E(H)\}$. Clearly, $\deg_H(S)=|N_H(S)|$. 
For simplicity, we write $N_H(v)$, $N_H(uv)$ and $\deg_H(v)$, $\deg_H(uv)$ when $S = \{v\}$ or $\{u, v\}$.
For $U\subseteq V(H)$ and $S\in\binom{V(H)}{d}$, we let
\[\deg_{H}(S,U)=\left|\left\{S' \in N_H(S): S' \subseteq U \right\}\right|.\]

Given a positive integer $t$, we denote the complete bipartite graph on the parts $X$ and $Y$ with $|X|=2$ and $|Y|=t$ by $K_{2,t}$. 
Given a graph $G$ and its \emph{$k$-expansion} $G^{(k)}$, we refer to the original vertices of $G$ as \emph{anchor vertices}.
Anchor vertices are retained in $G^{(k)}$ and may participate in multiple hyperedges.
The newly introduced vertices are called \emph{expansion vertices}; each expansion vertex is specific to a single edge and occurs in exactly one hyperedge of $G^{(k)}$.

Recall that a $k$-uniform $\ell$-tree is a $k$-graph that admits an ordering of its edges $e_1, \dotsc, e_m$ such that each $e_i$ has $\ell$ vertices in a previous edge in the ordering, together with $k-\ell$ new vertices.
Any ordering of the edges in a $k$-uniform $\ell$-tree satisfying this property will be called a \emph{valid ordering}.

Given a tree $T$,
we use the notation $(T, x)$ to denote a rooted tree $T$ with root $x$.
For $v \in V(T)$, we write $T(v)$ for the subtree of $T$ including all descendants of $v$ in $(T, x)$ including $v$, and we use $T^{(k)}(v)$ for the $k$-expansion of $T(v)$.
Let $D(v)$ denote the children of $v$ in $(T, x)$. 
Given a $k$-expansion hypertree $T^{(k)}$,
a \emph{leaf} of $T^{(k)}$ is a $(k-1)$-set 
$\{u_1, u_2, \ldots, u_{k-1}\}$ such that there is an edge $e\in E(T^{(k)})$ containing the $(k-1)$-set and $\deg_{T^{(k)}}(u_i)=1$ for each $i\in [k-1]$. Each vertex in a leaf is called a \emph{leaf vertex} of $T^{(k)}$. 

Given a $k$-graph $H$ and a subgraph $H_0\subset H$, we use $H - V(H_0)$ to denote the remaining subgraph after vertex deletion, and use $H - H_0$ to denote the remaining subgraph after edge deletion and removing any resulting isolated vertices. For two $k$-graphs $H_1$ and $H_2$, an \emph{embedding} from $H_1$ into $H_2$ is an injective edge-preserving mapping $\phi: H_1\to H_2$. Given such a mapping $\phi$, we write $\phi (X):=\{\phi (x): x\in X\}$ for $X\subseteq V(H_1)$ and use $\phi (H_1)$ for the copy of $H_1$ in $H_2$. 
 
We use the hierarchy $x\gg y$ to mean that for $x>0$, there exists $y_0>0$ such that for all $y\leq y_0$ the subsequent statements hold.  Hierarchies with more constants are defined analogously, and should always be read from left to right. 
To simplify presentation, we omit floor and ceiling whenever they are not crucial and treat all large numbers as integers when this does not affect our arguments.

The following simple and well-known proposition will be useful, which allows us to compare minimum $d$-degrees in a $k$-graph, for different values of $d$.

\begin{proposition}\label{degree}
Let $1\le d'\le d <k$ and $H$ be a $k$-graph on $n$ vertices. 
If $\delta_{d}(H) \ge x\binom{n-d}{k-d}$ for some $ 0\le x \le 1 $, then $\delta_{d'}(H) \ge x\binom{n-d'}{k-d'}$.
\end{proposition}

\subsection{Sketch of the proof: switchings} \label{ssection:sketch-gadgets}
Crucial to our proofs is the concept of switchings, which will be used to perform the absorption sketched in \Cref{ssection:intro-sketch}.
Given a $k$-expansion hypertree $T^{(k)}$, a $k$-graph $H$ and a partial embedding $\phi: T^{(k)} \rightarrow H$, we will consider `local switchings', meaning that we will modify the embedding $\phi$ in few vertices as to cover a different set of vertices.
Formally, let $X = \phi(V(T^{(k)}))$ be the image of $\phi$.
Given $X_1 \subseteq X$, $X_2 \subseteq V(H) \setminus X$, and a partial embedding $\phi': T^{(k)} \rightarrow H$, say that $(\phi, \phi')$ is an \emph{$(X_1, X_2)$-switch} if $\phi'(V(T^{(k)})) = (X \setminus X_1) \cup X_2$, and $\phi'$ coincides with $\phi$ in $X \setminus X_1$.
Thus, after performing an $(X_1, X_2)$-switch (i.e., replacing $\phi$ with $\phi'$), the vertices of $X_1$ are not used in the embedding anymore, and the vertices of $X_2$ are now used by the embedding.
This allows us to use the vertices from $X_1$ to embed other types of vertices; this is well-suited to build absorbers.

In our setting, we will initially embed a small subtree $T_0^{(k)} \subseteq T^{(k)}$ with a partial embedding $\phi$ that can be used repeatedly to perform local switches.
To ensure that this can indeed be done many times, we will need to ensure that $\phi$ uses a suitable set of edges.
For instance, suppose that we want to perform an $(X_1, X_2)$-switch.
Suppose that $Y \subset V(T^{(k)})$ is such that $\phi(Y) = X_1$.
Then to get an $(X_1, X_2)$-switch $(\phi, \phi')$, we would need that $\phi'(Y) = X_2$, and for every edge $e, e' \in E(T^{(k)})$ with $e\cap Y=\emptyset$ and $e'\cap Y\neq \emptyset$, we would need that both $\phi(e)=\phi'(e)$ and $\phi(e' \setminus Y ) \cup \phi'(e' \cap Y)$ are edges in $H$.
We will consider two basic instances of this.
Suppose $T$ is a tree and $T^{(k)}$ is its $k$-expansion.

\begin{enumerate}
    \item Let $e \in E(T^{(k)})$ with anchor vertices $u_1, u_2 \in e$ and let $y \in e$ be an expansion vertex.
    Suppose $D \subseteq H$ is a `diamond', i.e., two $k$-edges intersecting in $k-1$ vertices.
    Let $f_1, f_2$ be the two edges of $D$, let $x_1 \in f_1, x_2 \in f_2$ be the two vertices of degree one in $D$.
    We will say that $D$ is an \emph{$(x_1, x_2)$-diamond}.
    
    Suppose then that $\phi: T^{(k)} \rightarrow H$ is a partial embedding such that $\phi(e) = f_1$, $\phi(y) = x_1$, and $x_2 \notin \phi(V(T^{(k)}))$.
    Then we obtain $\phi'$ such that $(\phi, \phi')$ is an $(\{x_1\}, \{x_2\})$-switch just by changing the image of $y$ from $x_1$ to $x_2$. 
    
    \item Now suppose that $u \in V(T)$ is a vertex of degree $t$, with neighbors $v_1, \dotsc, v_t \in V(T)$.
    For each $i \in [t]$, let $e_i \in E(T^{(k)})$ be the edge that contains the anchor vertices $u$ and $v_i$.
    Let $J \subseteq H$ be the $k$-expansion of $K_{2,t}$, say, with vertex classes $\{x_1, x_2\}$ and $\{y_1, \dotsc, y_t\}$, and for each $\ell \in [2]$ and $i \in [t]$,  let $f_{\ell, i} \in E(J)$ be the edge containing $x_\ell$ and $y_i$.
    For each $\ell \in [2]$, let $X_\ell := \bigcup_{i \in [t]} f_{\ell,i} \setminus \{ y_i \}$.
    
    Suppose $\phi: T^{(k)} \rightarrow H$ is a partial embedding such that $\phi(u) = x_1$, $\phi(v_i) = y_i$ and $\phi(e_i) = f_{1,i}$ for each $i \in [t]$, and such that $\bigcup_{i \in [t]} f_{2,i} \setminus \{ y_i \}$ is disjoint from $\phi(V(T^{(k)}))$.
    Then we obtain $\phi'$ such that $(\phi, \phi')$ is an $(X_1, X_2)$-switch by making $\phi'(u) = x_2$ and $\phi'(v_i) = y_i$ and $\phi'(e_i) = f_{2,i}$ for each $i \in [t]$.
\end{enumerate}

We will obtain more complicated gadgets by combining many instances of these two `basic' switches.

\subsection{\texorpdfstring{Sketch of the proof of \Cref{theorem:main}}{Sketch of the proof of Theorem 1.1}} \label{ssection:sketch-main}
Now we consider the setting of \Cref{theorem:main}.
Let $H$ be an $N$-vertex $k$-graph with 
\[\delta_d(H) \geq \left(\max\left\{\frac{1}{2}, \deltaLHC{k}{d}\right\} + \eps\right)\binom{N-d}{k - d},\] 
and let $T$ be an $n$-vertex tree with 
$\Delta_1(T) \le \Delta$ and $N=(k-1)n-k+2$.
The goal is to embed $T^{(k)}$ in $H$.

Our absorbers for this proof will be built as follows.
Suppose $(w, w_1, \dotsc, w_{k-1})$ is a $k$-tuple where $w$ is a vertex in an embedding of a partial $k$-expansion hypertree where we need to place a new leaf, and $w_1, \dotsc, w_{k-1}$ are unused vertices by the embedding.
An absorbing structure for this tuple will be found first by finding an edge $e$ that contains $w$, say $e = wv_1 \dotsc v_{k-1}$.
Then, we find, for every $i \in [k-1]$, a $(w_i, v_i)$-diamond (this step crucially uses that the normalized minimum degree condition is larger than $1/2$).
Say the edges of the diamond are $f^w_i$ containing $w_i$ and  $f^v_i$ containing $v_i$.
If a partial embedding contains the edges $f^v_i$, then by performing switchings we can cover the vertices $w_i$ while at the same time we `free' the vertices $v_i$, after this step we can add the edge $e$ to extend our embedding.

Having the description of our absorbers, now we can describe our strategy.
We start by partitioning $T$ into three subgraphs $T_1, T_2, T_3$ such that
\begin{itemize}
	\item[(i)] $T_1$ and $T_2$ are subtrees of $T$;
	\item[(ii)] $|V(T_1)| \approx \nu n $ and $ |V(T_3)| \approx \alpha n$ for some $ 0 < \alpha \ll \nu \ll \eps$;
	\item[(iii)] $T_2$ is obtained from $T - T_1$ by removing `leaves' one by one.
\end{itemize}

First, we build the absorbing structures.
We find an embedding of $\smash{T_1^{(k)}}$ in $H$, alongside constructing a collection $\mathcal{A}$ of $ \beta n$ disjoint absorbing tuples for some $\nu \gg \beta\gg \alpha$.
These substructures possess two critical properties:
(i) as described before, each absorbing tuple will be capable to absorb $k-1$ vertices (one leaf) at a time, and so $\mathcal{A}$ can absorb any set of $(k-1)\alpha n$ vertices `one by one';
(ii) the embedding of $\smash{T_1^{(k)}}$ covers all members of $\mathcal{A}$.

Secondly, we embed an almost spanning $k$-expansion hypertree.
We find an embedding of $\smash{T_2^{(k)}}$ together with the embedding of $\smash{T_1^{(k)}}$ covering all but $(k-1)\alpha n$ vertices of $H$.
For this, we use \emph{blow-up covers}.
First, we cover almost all vertices with pairwise vertex-disjoint $(\log n)^c$-blow-ups of short loose cycles for some constant $c>0$ (this step uses the degree condition being above the threshold for loose Hamilton cycles).
Next, we decompose $\smash{T_2^{(k)}}$ into some small subtrees of nearly equal size, each of order $o((\log n)^c)$.
Then we embed subtrees, one by one, in the cycle blow-ups.
In order to connect all subtrees, we will choose a reservoir (a very small vertex set $R$) before the embedding of $\smash{T_1^{(k)}}$.
This will ensure that we can use vertices of $R$ to connect every two distinct subtrees.

Having done the last two steps, we can finish by using the absorbers. 
Since $T_3$ is spanned by at most $ \alpha n$ vertices in a valid ordering of  $T-T_1$, we can use $\mathcal{A}$ to complete the embedding of $T^{(k)}$.

\subsection{\texorpdfstring{Sketch of the proof of \Cref{theorem:main-even}}{Sketch of the proof of Theorem 1.2}} \label{ssection:sketch-even}
Our proof for \Cref{theorem:main-even} will also use absorbing and an `almost spanning' step as explained before; in fact we can use the same methods for the `almost spanning' step as before.
This is because the codegree condition is above $1/3$, well above the threshold for the existence of loose Hamilton cycles.
The main difference is the construction of the absorbers: here the minimum codegree conditions are below a half, so we cannot easily find diamonds between every pair of vertices.

The first observation is that to build an absorber for $(w, w_1, \dotsc, w_{k-1})$ as we did before ($w$ is a vertex in an embedding of a partial $k$-expansion hypertree and needs a leaf, and $w_i$ are unused), we can use a weaker property than connecting vertices with one diamond.
Given vertices $u, v$, a \emph{$(u, v, \ell)$-diamond-chain} is given by a sequence of distinct vertices $x_0, \dotsc, x_\ell$ with $x_0 = u$, $x_\ell = v$, and a collection of $\ell$ diamonds $D_1, \dotsc, D_\ell$ such that for each $1 \leq i \leq \ell$, $D_i$ is an $(x_{i-1}, x_i)$-diamond (see \Cref{ssection:vertexgadgets}).
Note that by switching all the diamonds in a $(u, v, \ell)$-diamond-chain simultaneously, we obtain a local switch that allows us to cover either $u$ or $v$ with an embedding.
Thus in our setting we can try to find, instead of a $(w_i, v_i)$-diamond, a $(w_i, v_i, \ell)$-diamond-chain for all $i \in [k-1]$.
If we can indeed find (short) $(w_i, v_i, \ell)$-diamond-chains for all pairs of vertices $w_i, v_i$, we can build absorbers in this way and proceed in a similar way to the proof of \Cref{theorem:main}. 

The strategy changes when we cannot find the diamond-chains as described.
To understand the structure of the host $k$-graph $H$, we define an auxiliary \emph{diamond graph} $G = G(H)$, where two vertices $x, y$ are joined by an edge in $G$ if $H$ contains many $(x, y)$-diamonds.
If there are many paths between $x$ and $y$, we will be able to find $(x, y, \ell)$-diamond-chains between them for some $\ell\in \mathbb N$.
Otherwise, using our relative codegree condition exceeding $1/3$, we can show that $G$ follows a rigid structure: there is a vertex partition $\{A,B\}$ such that we can find many $(x, y, \ell)$-diamond-chains for $x, y \in A$, or for $x, y \in B$; but not necessarily for $x \in A$ and $y \in B$ (see \Cref{ssection:diamond}).

Thus we can assume that we have such a vertex partition $\{A, B\}$, and we will investigate the existence of certain `balancer' gadgets, which are $k$-expansions of $K_{2,t}$ as described before.
An structural analysis will ensure that in fact we can find many such gadgets for any constant $t$; and moreover we will have a lot of freedom to ensure they have a prescribed number of vertices in $A$ and $B$ (see \Cref{ssection:balancers}).
This will mean that by switching between the two possible states of such a balancer gadget, we can modify the leftover to ensure we move $m(t) \in \{1,2\}$ vertices between $A$ and $B$, where $m(t) = 2$ if $t$ is odd, and $m(t) = 1$ if $t$ is even.
However, even though the balancer gadgets are abundant in the host $k$-graph $H$, to be able to incorporate such gadgets in a partial embedding of a $k$-expansion tree, we also need a vertex of degree $t$ in the original tree $T$.
This explains the relevance of even-degree vertices in $T$: they are the only ones that allow us to find a balancer gadgets moving exactly one vertex across $A$ and $B$.

Our construction of absorbers gadgets also needs to incorporate the structure given by the partition $\{A,B\}$.
We define a function $\pi: V(H) \rightarrow \{0,1\}$ where for every $w \in V(H)$, we have that $w$ has a majority of $(k-1)$-edges $W$ in its neighbors such that $|A \cap W| \equiv \pi(w) \bmod 2$.
By combining vertex-switchers and balancers, we will be able to build absorbers for tuples $(w, w_1, \dotsc, w_{k-1})$ as before, but only in the case where $|\{ w_1, \dotsc, w_{k-1}\} \cap A | \equiv \pi(w) \bmod 2$ (see \Cref{section:complexgadgets}).
This is enough to carry the absorbing procedure until the end, as long as from our initial absorbing step the leftover vertices are somewhat equally split between $A$ and $B$.
This last property is achieved by using the balancer gadgets before using any of the absorbers.
Finally, the special gadget built with the even-degree vertex will be used in the very last absorbing step, if necessary, to change the distribution of the last $k-1$ vertices to be embedded between $A$ and $B$.

\section{The parity obstruction} \label{section:parity}
Recall that a $k$-uniform loose cycle is the $k$-expansion of a graph cycle, and such a loose cycle $C \subseteq H$ is a \emph{loose Hamilton cycle} if $V(C) = V(H)$.
Also recall that if $G$ is a graph, then its $k$-expansion $G^{(k)}$ has precisely $|V(G^{(k)})| = |V(G)|+(k-2)|E(G)|$ vertices: this explains the divisibility requisites in the next two definitions.

\begin{definition}[Asymptotic $d$-degree threshold for loose Hamilton cycles]\label{def:LHC-threshold}
Let $k > d \geq 1$ be integers and $n\in \mathbb N$.
For $n$ divisible by $k-1$, let $\deltaLHC{k}{d}(n)$ be the maximum value of $\delta_d(H)$ taken over all $n$-vertex $k$-graphs $H$ that do not have a loose Hamilton cycle.
The \emph{asymptotic minimum $d$-degree threshold for loose $k$-uniform Hamilton cycles} is
\[ \deltaLHC{k}{d} := \limsup_{n \rightarrow \infty} \frac{\deltaLHC{k}{d}(n)}{\binom{n-d}{k-d}}. \]
\end{definition}

\begin{definition}[Asymptotic $d$-degree threshold for $k$-expansion hypertrees] \label{def:EST-threshold}
Let $k > d \geq 1$ be integers and $n, \Delta\in \mathbb N$.
For $n$ such that $n - 1$ is divisible by $k-1$, let $\deltaEST{k}{d}(n, \Delta)$ be the maximum value of $\delta_d(H)$ taken over all $n$-vertex $k$-graphs $H$ that do not contain some $k$-expansion hypertree $T^{(k)}$ with $\Delta_1(T^{(k)}) \leq \Delta$.
The \emph{asymptotic minimum $d$-degree threshold for $k$-expansion hypertrees} $T^{(k)}$ with $\Delta_1(T^{(k)}) \leq \Delta$ is
\[ \deltaEST{k}{d} := \limsup_{\Delta \rightarrow \infty} \limsup_{n \rightarrow \infty} \frac{\deltaEST{k}{d}(n, \Delta)}{\binom{n-d}{k-d}}. \]
\end{definition}

To prove $\deltaEST{k}{d} \geq \delta$, the following construction suffices: there exists a constant $\Delta \geq 1$  such that for any $n_0 \in \mathbb{N}$, we can find an integer $n \geq n_0$ with $k-1$ divides $n-1$, an $n$-vertex $k$-graph $H$ with $\delta_d(H) \geq (\delta + o(1)) \binom{n-d}{k-d}$, and an $n$-vertex $k$-expansion hypertree $T^{(k)}$ with $\Delta_1(T^{(k)}) \leq \Delta$,  for which $T^{(k)} \not\subseteq H$.

Let $T$ be a tree $\Delta(T) \leq 3$ that has only odd-degree vertices
(note that the Handshaking Lemma implies such trees only exist if $|V(T)|$ is even).
Such trees exist, even with bounded maximum degree at most $3$.
For instance, one could consider the tree $T$ obtained from a path $P$ where one leaf is attached to each vertex in the interior of $P$.
By the previous discussion, the following result then proves $\deltaEST{k}{d} \ge 1/2$.

\begin{theorem}\label{theorem:construct}
Let $k > d \geq 1$, let $n$ be even, and let $N = (k - 1)n-k + 2$.
Then there exists an $N$-vertex $k$-graph $H$ with $\delta_d(H) \geq (1/2 - o(1))\binom{N - d}{k - d}$ that does not contain the $k$-expansion of any $n$-vertex tree all of whose degrees are odd.
\end{theorem}

\begin{proof}
Let $H$ be a $k$-graph on $N$ vertices with vertex
partition $A \cup B = V(H)$, such that $\big| |A| - |B| \big| \leq 1$ and $|A|$ is even.
The edge set of $H$ consists of all $k$-element subsets of $V(H)$ that intersect $A$ in an odd number of vertices.
Observe that for any $1 \leq d \leq k-1$, we have
\[
\delta_d(H) = \left(\frac{1}{2}-o(1)\right) \binom{N-d}{k-d}.
\]

Now, let $T$ be any $n$-vertex tree with only odd-degree vertices.
We claim that $H$ does not contain $T^{(k)}$.
Suppose otherwise, so there exists an embedding $\phi: V(T^{(k)})\to V(H)$.
Fix an arbitrary leaf vertex $x\in V(T)$, and form a rooted tree $(T, x)$ by rooting $T$ at $x$.
For $v\in V(T)$, note that $D(v) \subseteq N_T(v)$ is the set of children neighbors of $v$ in $(T,x)$.
Let $F(v) \subseteq E(T^{(k)}(v))$ be the set of edges of $T^{(k)}$ corresponding to the expanded edges that are formed by expanding edges in $T$ of the form $vw$, with $w \in D(v)$.
Let $e_0\in E(T^{(k)})$ be the unique edge containing $x$.
For each $v \in V(T) \setminus \{x\}$, 
we define $B(v):=V(F(v))\setminus \{ v \} \subseteq V(T^{(k)})$; and for $v = x$ we define $B(v) = e_0 \subseteq V(T^{(k)})$.
Note that we have
\[ V(T^{(k)})=\bigcup_{v\in V(T)}B(v).\]

Since the degree of each vertex of $T$ is odd, $|D(v)|$ is even for all $v\in V(T)\setminus \{x\}$.
This implies that $|\phi(B(v))\cap A|$ is even for all $v\in V(T)\setminus \{x\}$.
On the other hand, we have that
$|\phi(B(x))\cap A| = |\phi(e_0) \cap A|$ is odd by the construction of $H$.
For distinct $v_1,v_2\in V(T)$, we have $B(v_1)\cap B(v_2)=\emptyset$, and therefore 
\[ |A| = |\phi(V(T^{(k)}))\cap A| = \sum_{v \in V(T)} |\phi(B(v))\cap A|, \]
and the last sum must be odd, by our previous observations, but this contradicts our choice of $|A|$.
Hence, $H$ does not contain $T^{(k)}$ as a subgraph, as desired.
\end{proof}

We generalize the construction above to deal with trees whose degrees all are $1 \bmod q$, for some positive integer $q \geq 2$.
Such trees $T$ must satisfy that $|V(T)|-2$ is divisible by~$q$.
Such trees exist with bounded maximum degree at most $q+1$, say by starting from a path $P$ and then adding $q-1$ leaves to each vertex in the interior of $P$.
The next result is an analogous version of \Cref{theorem:construct} for those trees.

\begin{theorem} \label{theorem:construct-q}
    Let $k > d \geq 1$ and $1 < q \leq k$.
    Let $n$ be such that $n-2$ is divisible by $q$, and let $N = (k-1)n - k + 2$.
    Then there exists an $N$-vertex $k$-graph $H$ with $\delta_{d}(H) \geq (1/q -o(1)) \binom{N-d}{k-d}$ that does not contain any $k$-expansion of any $n$-vertex tree with all degrees of value $1$ modulo $q$.
\end{theorem}

\begin{proof}
    Let $V$ be an $N$-vertex set, and partition it into $q$ parts $V_1, \dotsc, V_q$ of size as equal as possible.
    For every $v \in V$, let $c(v) = i$ if $v \in V_i$.
    For a set $S \subseteq V$, let $c(S) = \sum_{v \in S} c(v)$.

    Note that $c(V) = \sum_{i=1}^{q} i |V_i|$.
    If $c(V) \not\equiv 1 \bmod q$, we fix this partition.
    Otherwise, we move one vertex from $V_1$ to $V_2$.
    This increases $c(V)$ by exactly one. 
    In any case, we obtain a partition where every cluster has size between $\lfloor N/q \rfloor - 1$ and $\lceil N/q \rceil + 1$, and $c(V) \not\equiv 1 \bmod q$.
    
    Let $H$ be a $k$-graph on set $V$ as follows.
    Define the edge set of $H$ as the $k$-sets $S$ such that $c(S) \equiv 1 \bmod q$.
    Note that every $(k-1)$-set has degree $(1/q - o(1))N$, and therefore $\delta_d(H) = (1/q - o(1)) \binom{N-d}{k-d}$ holds for all $1 \leq d < k$.

    Now suppose that $T$ is an $n$-vertex tree all of whose degrees are $1$ modulo $q$.
    We claim that $H$ does not contain $T^{(k)}$.
    Suppose otherwise, so there exists an embedding $\phi: T^{(k)}\to H$.
    Fix an arbitrary leaf vertex $x \in V(T)$, and form a rooted tree $(T, x)$ by rooting $T$ at $x$.
    For $v\in V(T)$, let $D(v) \subseteq N_T(v)$ be the set of children neighbors of $v$ in $(T, x)$.
    Let $F(v) \subseteq E(T^{(k)}(v))$ be the set of edges of $T^{(k)}$ corresponding to the expanded edges that are formed by expanding the edges in $T$ of the form $vw$, with $w \in D(v)$.
    Let $e_0\in E(T^{(k)})$ be the unique edge containing $x$.
    For each $v \in V(T) \setminus \{x\}$, we define $B(v):=V(F(v))\setminus \{ v \} \subseteq V(T^{(k)})$; and for $v = x$ we define $B(v) = e_0 \subseteq V(T^{(k)})$.
    Note that $V(T^{(k)})$ is the vertex-disjoint union of $B(v)$ over all $v \in V(T)$.
    Since the degree of each vertex of $T$ is $1$ modulo $q$, we have $|D(v)| \equiv 0 \bmod q$, for all $v\in V(T)\setminus \{x\}$.
    This implies that, for all $v \in V(T) \setminus \{x\}$, we have 
    \[c(\phi(B(v))) =\sum_{e \in F(v)} c(\phi(e))  - |D(v)|c(\phi(v)) \equiv |D(v)| \equiv 0 \bmod q.\]
    On the other hand, we have that
    $c(\phi(B(x))) = c(\phi(e_0)) \equiv 1 \bmod q$, by the construction of $H$.
    We conclude that
    \[ 1 \equiv c(\phi(B(x))) \equiv \sum_{v \in V(T)} c(\phi(B(v))) \equiv c(V) \not\equiv 1 \bmod q, \]
    a contradiction.
    Hence $H$ does not contain $T^{(k)}$ as a subgraph, as desired.
\end{proof}

Note that \cref{theorem:construct} follows from \cref{theorem:construct-q} by setting $q=2$, so it was not strictly necessary to include it.
However, we decided to emphasize this case for its simplicity, and also because the $k$-graphs appearing in \cref{theorem:construct} (with a partition $\{A,B\}$ of the vertex set, only including the edges with odd intersection with $A$) form the basis for our analysis in the proof of \cref{theorem:main-even}.

\section{Tools} \label{section:tools}
We first recall the definition of the blow-up.
For $k\ge 2$, a {\it blow-up} $F^*$ of a $k$-graph $F$ is obtained by replacing each vertex $x\in V(F)$ by a non-empty vertex sets $V_x$ and each edge $e=\{x_1,\dots ,x_k\}\in E(F)$ by a complete $k$-partite $k$-graph on parts $V_{x_1}, \dots, V_{x_k}$. We refer to the sets $V_x$ as the \emph{clusters} of $F^*$.
An \emph{$m$-blow-up} of $F$ is a blow-up of $F$ where each cluster has size $m$, and we denote it by $F(m)$. For the $m$-blow-up of an edge, we have a classic `supersaturation' result (see, e.g.,~\cite[Corollary 2]{supersaturation}).

\begin{proposition}\label{supersa}
For all $k\geq 2, m\geq 1$ and $\eps'>0$, there are $\eps>0$ and $n_0\in \mathbb N$ such that every $k$-graph $H$ on $n\geq n_0$ vertices with $|E(H)|\geq \eps' n^k$ contains at least $\eps n^{km}$ copies of $K^{(k)}_k(m)$, where $K^{(k)}_k(m)$ denotes the $m$-blow-up of an edge.
\end{proposition}

\subsection{Trees in blow-ups of cycles}
Given an odd (graph) cycle $C$, the following result \cite[Lemma 6.6]{Joos-Kim-Kuhn-Osthus-19}
describes how to embed a bounded-degree tree into an $m$-blow-up of $C$ in a very balanced way.
Its proof is based on analyzing a random walk in a cycle.

\begin{lemma}
\label{lem:cycle-is-covered-by-tree-cite}
Let $n, \Delta, \ell\in \mathbb N$ with $1 > 1/\ell, 1/\Delta \gg 1/n$, where $\ell$ is odd.
Let $m = n/\ell + n/(\log n)^2$.
Let $T$ be a tree on $n$ vertices with $\Delta_1(T) \leq \Delta$
and let $C$ be a cycle of length $\ell$.
Then there is an embedding $\phi$ of $T$ into the $m$-blow-up $C(m)$ such that $\phi(T)$ contains at most $m$ edges between any two clusters of $C(m)$.
\end{lemma}

\cref{lem:cycle-is-covered-by-tree-cite} implies the following result for the embedding of $k$-expansion hypertrees into blow-ups of loose cycles.

\begin{lemma}\label{lem:cycle-is-covered-by-tree}
Let $k\ge 3$ and $n, \Delta, \ell\in \mathbb N$ with $1 > 1/\ell, 1/\Delta \gg 1/n$, where $\ell$ is odd.
Let $m = n/\ell + n/(\log n)^2$.
Let $T$ be a tree on $n$ vertices with $\Delta_1(T) \leq \Delta$, and let $H$ be the $m$-blow-up of a $k$-uniform loose cycle of length $\ell$.
Then there is an embedding $\psi$ of $T^{(k)}$ into $H$ such that for each cluster $X$ of $H$,
\[
\frac{n}{\ell}-\frac{((\ell-1)(k-1) - 1)n}{(\log n)^2} - (k-2) \le |V(\psi(T^{(k)}))\cap X|\le \frac{n}{\ell}+\frac{n}{(\log n)^2}.
\]
\end{lemma}

\begin{proof}
    Let $C = v_1 \dotsb v_\ell$ be a cycle on $\ell$ vertices and let $C(m)$ be its $m$-blow-up.
    By \cref{lem:cycle-is-covered-by-tree-cite}, there is an embedding $\phi: T \rightarrow C(m)$ such that $\phi(T)$ contains at most $m$ edges between any two clusters of $C(m)$.
    Next, let $T^{(k)}$ be the $k$-expansion of $T$, and for every edge $e \in E(T)$ let $X_e$ be the $(k-2)$-set of vertices, which together with $e$ forming an edge  in $T^{(k)}$, so that $V(T^{(k)}) = V(T) \cup \bigcup_{e \in E(T)} X_e$.
    Since a $k$-uniform loose cycle is precisely the $k$-expansion of a cycle, we can assume $C^{(k)}$ is the $k$-expansion of $C$ and for each $f \in E(C)$ there exists a $(k-2)$-set $Y_f$ such that $V(C^{(k)}) = V(C) \cup \bigcup_{e \in E(C)} Y_e$.
    Abusing notation, we can suppose that the clusters in $H = C^{(k)}(m)$ are the clusters of $C(m)$ for each $v \in V(C)$, together with extra $k-2$ new clusters $Y^1_f, \dotsc, Y^{k-2}_f$ for each edge $f \in E(C)$.
    We define an embedding $\psi$ of $T^{(k)}$ into $H$ by noting that, since $V(T) \subseteq V(T^{(k)})$ and $V(C) \subseteq V(C^{(k)})$, we can assume $\psi(v) = \phi(v)$ for each $v \in V(T)$; and for each $e = \{x, y\} \in E(T)$, we map $\phi(X_e)$ by assigning one vertex to each the $k-2$ clusters $Y^1_f, \dotsc, Y^{k-2}_f$, where $f = \{\phi(x), \phi(y)\}$.    
    Note that each the number of edges mapped to each $Y^1_f, \dotsc, Y^{k-2}_f$ is precisely the number of edges mapped via $\phi$ between the clusters corresponding to the two vertices of $f$; so the choice of $\phi$ ensures that at each cluster of $H$ indeed receives at most $m$ vertices.
    The upper bound in the claim follows.
    The lower bound follows by noting that $|V(T^{(k)})| = n(k-1) - k+2$ and there are exactly $\ell(k - 1)$ clusters in $H$, so each cluster receives at least $|V(T^{(k)})| - (\ell(k-1) - 1)m$ vertices.
\end{proof}

\subsection{Blow-up covers}
Given a property $\mathcal P$ and a $k$-graph $H$ satisfying $\mathcal P$, we are interested in which subgraphs of $H$ inherit the property $\mathcal P$.
Following the work of  Lang and the second author \cite{Lang-23, Lang-Sanhueza-2024}, this is formalized in terms of the property graph.

\begin{definition} [Property graph]
For a $k$-graph $H$ and a family $\mathcal P$ of $s$-vertex $k$-graphs, the {\it property graph}, denoted by $P^{(s)}(H; \mathcal P)$, is the $s$-graph on vertex set $V(H)$ with an edge $S \subseteq V(H)$ whenever the induced subgraph $H[S]$ satisfies $\mathcal P$, that is $H[S] \in \mathcal P$.
\end{definition} 

Given a property $\mathcal P$ and a $k$-graph $H$ satisfying $\mathcal P$, if the minimum vertex degree of the property graph $P^{(s)}(H; \mathcal P)$ goes beyond the minimum vertex degree threshold that forces a perfect matching, then $H$ can be almost covered by vertex-disjoint blow-ups of subgraphs of $H$ that inherit the property $\mathcal P$. 
The version we use is a direct consequence of~\cite[Lemma 7.7]{Lang-Sanhueza-2024}, and allows us to select the clusters of these blow-ups to have a size that tends to infinity with $n$.

\begin{lemma}\label{lem:blow-up-tiling}
Let $n,m,s,k\in \mathbb N$, $ 1/k, 1/s \gg c, \eta, \mu\gg 1/n $ and $1\le m \le (\log n)^c$. For every $s$-vertex $k$-graph property $\mathcal P$ and $k$-graph $H$ on $n$ vertices with 
$$\delta_1(P^{(s)}(H; \mathcal P)) \geq \left( 1 - \frac{1}{s} + \mu \right) \binom{n-1}{s - 1},$$ 
all but at most $ \eta n$ vertices of $H$ may be covered with pairwise vertex-disjoint $m$-blow-ups of members of $\mathcal P$.
\end{lemma}
 
Let $\mathcal P(\delta, \eps, d, k)$ be a family of $k$-graphs $H$ with $\delta_d(H) \geq (\delta + \eps)\binom{|V(H)|-d}{k - d}$.
The next lemma \cite[Lemma 4.9]{Lang-23} shows that $\mathcal P(\delta, \eps, d, k)$ satisfies a local inheritance principle.

\begin{lemma}\label{lem:Inheritance-mini-degree}
For $1/k, \eps \gg 1/s \gg 1/n$, let $H$ be an $n$-vertex $k$-graph such that $H\in \mathcal P(\delta, \eps, d, k)$.
Then the property $s$-graph 
$P := P^{(s)}(H; \mathcal P(\delta, \eps/2, d, k))$
satisfies
$$\delta_1(P) \geq \left( 1 -  e^{-\sqrt{s}}\right) \binom{n-1}{s - 1}.$$
\end{lemma}

Using \cref{lem:blow-up-tiling} and \cref{lem:Inheritance-mini-degree}, we can get the following result.

\begin{theorem}\label{thm:blow-up-cycle-tiling}
Let $k,d,\ell,m,n$ be positive integers with $k>d \ge 1$, $1/k, \eps \gg 1/\ell \gg c, \eta \gg 1/n$, and $m=(\log n)^c$. Let $H$ be an $n$-vertex $k$-graph with 
\[\delta_d(H) \geq \left( \deltaLHC{k}{d} + \eps\right)\binom{n-d}{k - d}.\]
Then all but at most $ \eta n$ vertices of $H$ can be covered with pairwise vertex-disjoint $m$-blow-ups of loose cycles of length $\ell$.
\end{theorem}
\begin{proof}
Given $k>d \ge 1$ and $\eps >0$, we choose  
\[
1/k, \eps \gg 1/\ell \gg c, \eta \gg 1/n,
\]
and let $s := \ell(k-1)$.
We can assume $\ell$ is large enough so that $e^{-\sqrt{s}} < 1/(2s)$.
Also, let $m=(\log n)^c$ and $\delta := \deltaLHC{k}{d}$.
Consider the property $s$-graph 
$P := P^{(s)}(H; \mathcal P(\delta, \eps/2, d, k))$.
By \cref{lem:Inheritance-mini-degree}, we have 
 \[\delta_1(P) > (1 - e^{- \sqrt{s}}) \binom{n-1}{s - 1}.\]
Furthermore, by \cref{lem:blow-up-tiling} and the choice of $s$ and $n$,
this gives us a collection of $s$-vertex $k$-graphs $F_1, \dots, F_{t} \in \mathcal P(\delta, \eps/2, d, k)$ such that there are pairwise vertex-disjoint $m$-blow-ups $F_i(m)$ of each $F_i$, which together cover all but $\eta n$ vertices of $H$. 

Note that each $s$-vertex $k$-graph $F_i\in \mathcal P(\delta, \eps/2, d, k)$ satisfies that $\delta_d(F_i) \geq ( \deltaLHC{k}{d} + \frac{\eps}{2})\binom{s-d}{k - d}$.
Since $s$ is sufficiently large and divisible by $k-1$, we deduce that $F_i$ has a loose Hamilton cycle for all $1\le i\le t$.
This implies that each $F_i(m)$ contains an $m$-blow-up of a loose cycle of length $\ell$.
\end{proof}

\subsection{Partitioning trees}
The following result \cite[Proposition 6.5]{Joos-Kim-Kuhn-Osthus-19} shows how to decompose a rooted tree into subtrees of controlled size.
Such a decomposition can be found by a simple greedy algorithm.

\begin{lemma}
\label{partition-tree}
Suppose integers $n, \Delta >1$ and $n \geq m' \geq 1$.
Then for any rooted tree $(T,x)$ on $n$ vertices with $\Delta_1(T) \leq \Delta$,
there exists a collection $\mathcal{T}$ of pairwise vertex-disjoint rooted subtrees such that
\stepcounter{propcounter}
\begin{enumerate}[label = {{\rm (\Alph{propcounter}\arabic{enumi})}}]
    \item $T_i \subseteq T(x_i)$ for every $(T_i, x_i) \in \mathcal{T}$,
    \item $m' \leq |V(T_i)| \leq 2 \Delta m'$ for every $(T_i, x_i) \in \mathcal{T}$, and
    \item $\bigcup_{(T_i, x_i) \in \mathcal{T}} V(T_i) = V(T)$.
\end{enumerate}
\end{lemma}

\subsection{Connectivity and reservoir}
For our proofs, we will wish to join two vertices $x_1, x_2$ in the host $k$-graphs with a short loose path (i.e., the expansion of a graph path).
Given a $k$-graph $H$ and $x_1, x_2 \in V(H)$, say that a loose path $P' \subseteq H$ is an \emph{$(x_1, x_2)$-path} if there is a graph path $P$ with endpoints $x_1, x_2$ and $P' = P^{(k)}$.
The \emph{length} of a path is its number of edges.
Note that if $\delta_d(H) > 0$ and $d \geq 2$, then $x_1$ and $x_2$ are contained in a common edge, which is an $(x_1, x_2)$-path of length one.
If, instead, we only have lower bounds on $\delta_1(H)$, we need a longer loose path to connect $x_1$ and $x_2$.
The following lemma achieves this using a loose path of length two.
Note that, by \eqref{equation:lowerbound-deltaLHC}, we can apply this lemma in $k$-graphs whenever $\delta_1(H) \geq (\deltaLHC{k}{d}+ \eps)\binom{n-1}{k-1}$.

\begin{lemma} \label{lemma:shortpath}
    Let $k \geq 2$, and $1/k, \eps \gg \delta, 1/n$.
    Let $H$ be an $n$-vertex $k$-graph with $\delta_1(H) \geq (2^{-k+1} + \eps)\binom{n-1}{k-1}$.
    For every pair of distinct vertices $x_1, x_2 \in V(H)$, there are at least $\delta n^{2k-3}$ many $(x_1, x_2)$-paths in $H$ of length two.
\end{lemma}

To prove this, it suffices to show the existence of edges $e, f \in E(H)$ with $x_1 \in e \setminus f$, $x_2 \in f \setminus e$, and $e \cap f \neq \emptyset$ (we can then conclude by supersaturation, \cref{supersa}).
To do so, note that by the Kruskal-Katona theorem the link graph of $x_1, x_2$ both must span more than $n/2$ vertices, from which we can easily find the required edges.
See~\cite[Proposition 9.4]{Alvarado-Kohayakawa-Lang-Mota-Stagni-2023} for a similar argument.

When connecting two subtrees into a large one, we want to ensure that the vertices used for the connection all come from a small set, called \emph{reservoir}, that is disjoint to the subtrees.
Such a reservoir has appeared in various forms in the literature (e.g.,~\cite{Bub-Han-Schacht-13, Daniela2010Hamilton, Pavez-Matias-Nicolas-Stein-24, Pehova-Petrova-24}) and can be proved via a standard probabilistic analysis.
In our case, this needs to be combined with~\Cref{lemma:shortpath}.
We omit the details and refer the reader to~\cite[Lemma 6]{Bub-Han-Schacht-13} and~\cite[Lemma 8.1]{Daniela2010Hamilton} for similar proofs.

\begin{lemma}\label{res-lem}  
Suppose that $k \geq 2$, and $1/k,\eps\gg \delta \gg 1/n$.
Let $H$ be a $k$-graph on $n$ vertices satisfying 
\[
\delta_d(H)\ge \begin{cases} 
(\frac{1}{2^{k-1}}+\eps) \binom{n-d}{k-d} & \text{for } d=1, \\
\eps \binom{n-d}{k-d} & \text{for }  2\le d\le k-1.
\end{cases}
\]
Also, let $\ell_1 = 2$ and $\ell_d = 1$ for all $2\le d\le k-1$.
Then there exists $R\subseteq V(H)$ of size $\delta n$ such that for every pair of vertices $x_1, x_2 \in V(H)$, there are at least $\delta |V(R)|^{\ell_d(k-1)-1}$
many $(x_1, x_2)$-paths $P$ of length $\ell_d$ with $V(P) \setminus \{x_1, x_2\} \subseteq R$.
\end{lemma}

\section{Embedding almost spanning expansion hypertrees} \label{section:almost}
In this section, we prove the existence of any $(1-\alpha)n$-vertex $k$-expansion hypertree of bounded-degree in an $n$-vertex $k$-graph above the loose Hamilton cycle threshold.
This will be used in the proof of our two main theorems.

\begin{theorem}\label{almost-cover}
    Suppose $k > d \geq 1$, $\Delta \in \mathbb N$ and $ \eps, \alpha >0$, there exists an $n_0\in \mathbb N$ such that the following holds. Let $H$ be a $k$-graph on $n\geq n_0$ vertices with
    \[
    \delta_d(H) \geq \left(\deltaLHC{k}{d} + \eps \right) \binom{n-d}{k-d}
    \]
    and let $w \in V(H)$. Let $T^{(k)}$ be the $k$-expansion of any tree $T$ with $\Delta_1(T) \leq \Delta$ and $|V(T^{(k)})| \leq (1 - \alpha)n$, and let $v \in V(T)$. 
    Then there is an embedding $\phi: T^{(k)}\to H$ with $\phi (v) = w$.
\end{theorem}

\begin{proof}
\noindent \emph{Step 1: Embedding the root and very small hypertrees.}
Let $n' = |V(T)|$, and 
for the given $v \in V(T)$ we let $N_{T} (v) := \{v_1, v_2, \ldots, v_{r_0}\}$.
For each $i \in [r_0]$, we denote the edge between $v$ and $v_{i}$ by $e_i$. Note that $r_0 \leq \Delta$ and $|V(T^{(k)})| = (k-1) n' - k + 2 \leq (1-\alpha)n$.
We need to find an embedding $\phi$ of $T^{(k)}$ in $H$ such that $\phi (v) = w$.
For the remainder of the proof, we choose $\delta, \ell, c, \eps'$ so that
\begin{equation}\label{equa-1}
1/k, 1/\Delta, \eps, \alpha\gg  \delta, 1/\ell\gg c,\eta\gg \eps'\gg 1/n,  
\end{equation}
and furthermore $\ell$ is an odd integer.

We first choose a reservoir.
By \cref{res-lem}, there exists a subset $R \subseteq V(H) \setminus \{w\}$ of size $\delta n$ such that for every pair $x, y \in V(H)$, we have (i) $\deg_H (\{x, y\}, R) \ge \delta (\delta n)^{k-2} / 4$ when $d \geq 2$; (ii) there are at least $\delta (\delta n)^{2k - 3} / 4$ many $(x, y)$-paths $P$ of length 2 with $V(P) \setminus \{x, y\} \subseteq R$ when $d = 1$.
Let $H_0 := H - R$ and $|V(H_0)|:=n_0$. The choice of $\eps\gg \delta$ ensures that $H_0$ satisfies
\[\delta_d(H_0) \geq \left(\deltaLHC{k}{d} + \frac{\eps}{2} \right)\binom{n_0 - d}{k - d}.\] 

Let $m := (\log n_0)^c$.
Consider the rooted tree $(T,v)$ and let
\[
T_0 :=T - \{V(T(v_i)): i \in [r_0] ~~ \text{and} ~~ |V(T(v_i))| \geq \eps' m \}.
\]
Namely, $T_0$ contains $v$ and some small subtrees $T(v_i)$ with $|V(T(v_i))| < \eps' m$.
Clearly, we have $1 \leq |V(T_0)| \leq \Delta \eps' m$.
Thus we can embed $v$ into vertex $w$ and embed $T_0^{(k)}$ into $H_0$ greedily, by adding one edge at a time. Indeed, at every step, the minimum $d$-degree of the remaining subgraph of $H_0$ is at least 
\[
\left( \deltaLHC{k}{d} + \frac{\eps}{2} \right) \binom{n_0 - d}{k - d} - \Delta (k - 1) \eps' m \binom{n_0 - d}{k - d - 1} > \left( \deltaLHC{k}{d} + \frac{\eps}{4} \right) \binom{n_0 - d}{k - d},
\]
which implies that the next edge in the embedding can be chosen greedily.
We denote by $\phi_0 : T_0^{(k)} \to H_0$ the embedding obtained in this way. \medskip

\noindent \emph{Step 2: Partitioning into small subtrees and allocation.}
Let $T_b := T - V(T_0)$.
Note that $T_b$ is a forest, and it contains at most $\Delta$ subtrees, each of which must have more than $\eps' m$ vertices. 
Rename the subtrees by $T_1, T_2, \ldots, T_{r_0'}$, where $1 \leq r_0' \leq r_0 \leq \Delta$. 
Let $H_1 := H_0 - V(\phi_0 (T_0^{(k)}))$ and $|V(H_1)|:=n_1$. The choice of $\eps\gg \delta\gg \eps'$ ensures that $H_1$ satisfies
\[\delta_d(H_1) \geq \left(\deltaLHC{k}{d} + \frac{\eps}{4} \right)\binom{n_1 - d}{k - d}.\] 

Apply~\cref{thm:blow-up-cycle-tiling} to $H_1$ and $m$, to obtain that all but at most $\eta n_1$ vertices of $H_1$ can be covered with pairwise vertex-disjoint $m$-blow-ups of loose cycles of length $\ell$.
We denote those blow-ups by $C_1(m),C_2(m), \dots, C_r(m)$.
Note that we have $r \leq \frac{n_1} {(k- 1) \ell m}\le \frac{(1 - \delta) n} {(k- 1) \ell m}$.

Now we decompose the trees in $T_b$ into small-sized trees.
Recall that $|V(T_i)| \geq \eps' m$ holds for each $i \in [r_0']$.
We decompose $T_b$ as follows.
For each $i \in [r_0']$, if $|V(T_i)| > 2 \Delta \eps' m$, then we invoke \cref{partition-tree} with parameters $\Delta$ and $\eps'm$, to obtain a decomposition of $T_i$.
Otherwise, we simply incorporate the whole tree $T_i$ to the decomposition.
Let $r'$ be the total number of subtrees in this decomposition.
We have found a collection 
$\{(T_{b,i}, x_i)\}_{1 \le i \le r'}$ of pairwise vertex-disjoint subtrees such that
\stepcounter{propcounter}
\begin{enumerate}[label = {{\rm (\Alph{propcounter}\arabic{enumi})}}]
     \item $T_{b,i} \subseteq T(x_i)$ for each $i \in [r']$;
     \item $\eps' m \leq |V(T_{b,i})| \leq 2 \Delta \eps' m$ for each $i \in [r']$;
     \item $\bigcup_{i \in [r']} V(T_{b,i}) = V(T_b)$.
\end{enumerate} 
Looking at the $k$-expansion of those trees, there is a corresponding collection of $k$-expansion subtrees $\{ T^{(k)}_{b,i} \}_{1\le i\leq r'}$  
such that
\stepcounter{propcounter}
\begin{enumerate}[label = {{\rm (\Alph{propcounter}\arabic{enumi})}}]
    \item $V(T_{b,i}^{(k)}) \cap V(T_{b,i'}^{(k)}) = \emptyset$ for distinct $i, i' \in [r']$; 
    \item $(k-1) \eps' m - k
    \leq |V(T_{b,i}^{(k)})| \leq 2 \Delta (k - 1) \eps' m$ for each $i \in [r']$.
\end{enumerate}
Note that all the roots $x_i$ are anchor vertices in the $k$-expansion $T_b^{(k)}$.
In addition, we have $r'< \frac{|V(T^{(k)})|} {(k - 1) \eps' m - k} \leq \frac{(1 - \alpha)n} {k \eps' m }$.

Now, we allocate the vertices in
$\{(T^{(k)}_{b,i}, x_i)\}_{1 \le i \le r'}$ to $\{C_j(m)\}_{1 \le j \le r}$. For each $i \in [r']$,
let 
\[
    q_i:=\frac{|V(T_{b,i})|} {\ell} + \frac{|V(T_{b,i})|} {(\log |V(T_{b,i})|)^2},
\]
and $q:=\max\{q_i: i \in [r']\}$. Clearly, 
\[
    q\le \frac{2\Delta\eps'm}{\ell}+\frac{2\Delta\eps'm}{(\log(\eps'm))^2}.
\]  
We slice each $C_j (m)$ for $j \in [r]$ into certain sub-hypergraphs consisting of $C(q_i)$ for $i \in [r']$ and some $q'$-blow-up of loose cycle, where $q'\le q$. We will allocate the vertices in subtree $T_{b,i}^{(k)}$ to $C(q_i)$.
Since
\[
    \begin{split}
    & (k - 1) \ell \sum_{i=1}^{r'} \left( \frac{|V(T_{b,i})|}{\ell} + \frac{|V(T_{b,i})|}{(\log |V(T_{b,i})|)^2}\right) +  (k - 1) \ell  rq \\
    & \leq  (k - 1) |V(T_{b})| + \frac{(k - 1) \ell |V(T_b)|}{(\log (\eps' m))^2} +  \frac{(1-\delta) n}{m} \left( \frac{2 \Delta \eps' m}{\ell} + \frac{2 \Delta \eps' m}{(\log(\eps'm))^2} \right) \\
    & \leq (1 - \alpha) n + \frac{2 \Delta \eps' (1 - \delta) n}{\ell} + o(n) \\
    & < (1 - \delta)(1 - \eta) n,
    \end{split}
\]
where the last inequality follows from $\alpha\gg \delta\gg \eta\gg \eps'$, we can finish the allocation of the vertices in $\{T^{(k)}_{b,i}\}_{1\le i\le r'}$. 
 \medskip

\noindent \emph{Step 3: Connecting via the reservoir and finishing embedding.}
In the following, we establish the connection using vertices in $R$ and embed the subtree to $T^{(k)}_{b,i}$ to $C(q_i)$ to finish the embedding.

First assume $d=1$, we essentially use stars to establish connections. 
Without loss of generality, we may assume that the subtrees in $\{(T^{(k)}_{b,i}, x_i)\}_{1 \le i \le r'}$ are arranged in a valid ordering, which means that all anchor vertices in subtrees follow a valid ordering in the rooted tree $(T, v)$. 
In what follows, we embed the subtrees $T^{(k)}_{b, i}$ one by one and finish the connections simultaneously. 
We will consecutively define subtrees $T^{(k)}_0=Q_0\subset Q_1  \subset \dotsb \subset Q_{r'}=T^{(k)}$ such that $Q_i$ is the union of $Q_{i-1}$, $T^{(k)}_{b,i}$ and the edges connecting $Q_{i-1}$ and $T^{(k)}_{b,i}$ for each $i\in [r']$;
define embeddings $\phi_0, \dots, \phi_{r'}$ such that $\phi_i: Q_i\to H$ is an embedding of $Q_i$ in $H$, and define subsets $R=R_1\supset R_2\supset \dots \supset R_{r'}$ with $|R_{i}|\geq |R_{i-1}| - k\Delta$. We will use the vertices in $R_{i}$ to connect $Q_{i-1}$ and $T^{(k)}_{b,i}$.

Recalling that $\phi_0: Q_0\to H$ has been defined in Step 1. Now, suppose $i\in [r']$ is given and that the embedding $\phi_{i-1}$ of $Q_{i-1}$ has been defined.
We now define $\phi_i$ by extending $\phi_{i-1}$ as follows. Since the ordering is valid, the father of the root $x_{i}$ (denoted by $z_i$) has been embedded, and assume $\phi_{i-1}(z_i)=z'_i$. Let $N_{T_{b, i}} (x_{i}) := \{x_{i, 1}, x_{i, 2}, \ldots, x_{i, t_i}\}$ for some $t_i \leq \Delta$ and $S_{x_i}$ be the $(t_i+1)$-star induced on set $\{z_i, x_i, x_{i, 1}, x_{i, 2}, \ldots, x_{i, t_i}\}$ in $T$ centered at $x_i$.  
Note that 
\[
|R_{i}| \geq |R|- k\Delta(i-1) > \delta n/2,
\]
where the inequality follows from $i < r' \leq \frac{(1 - \alpha)n} {k \eps' m }$, $m = (\log n_1)^c$ and $\alpha \gg \delta\gg \eps'$.
Let $X_{i}$ be a cluster of $C(q_{i})$. 
We define an auxiliary $(2k-2)$-graph $F$ on the set $R_{i} \cup X_{i}$ with edge set $E(F) = \{e \subseteq R_{i} \cup X_{i}: H[e \cup \{z_i'\}] \text{ contains an } (z_i', w')\text{-path } P \text{ of length 2 with } V(P) \setminus \{z_i', w'\} \subseteq R_{i} \text{ and } w' \in X_{i}\}$. 
By \cref{res-lem}, we have $|E(F)| \geq \delta |X_i| |R_{i}|^{2k - 3}/2$. 
By an averaging argument, there exists a subset $R' \subseteq R_{i}$ with $|R'| = |X_i|$ such that $|E (F[R' \cup X_i])| \geq \xi |X_i|^{2k - 2}$, where $\xi$ is a constant satisfying $\xi \ll \delta$. 
Indeed, suppose otherwise, then we have 
\[
|E(F)| < \frac{\binom{|R_{i}|}{|X_i|} \times \xi |X_i|^{2k - 2}}{\binom{|R_{i}|}{|X_i|-(2k - 3)}} \leq \frac{\xi |X_i|^{2k - 2} |R_{i}|^{2k - 3}}{|X_i| (|X_i| - 1) \dotsb (|X_i| - 2k + 4)} < \frac{\delta |X_i| |R_{i}|^{2k - 3}}{2},
\] 
a contradiction.
Thus, by \cref{supersa}, the $(2k - 2)$-graph $F[R' \cup X_i]$ contains a $\Delta$-blow-up  $K_{2k - 2}^{(2k - 2)}(\Delta)$ of an edge in $F$. 
We denote the vertex clusters of $K_{2k - 2}^{(2k - 2)}(\Delta)$ by $Y_1, \ldots, Y_{2k - 2}$, where $|Y_j| = \Delta$ for each $j\in [2k-2]$ and $Y_{2k - 2} \subseteq X_i$.
Then $H[\cup_{j=1}^{2k-2}Y_j]$ contains a copy  $S'^{(k)}_{x_i}$ of the $k$-expansion  $S^{(k)}_{x_i}$ of $S_{x_i}$ such that $z'_i\in V(S'^{(k)}_{x_i})$. Thus, we can extend the embedding $\phi_{i-1}$ to $\phi_{i}$ such that  $\phi_i(V(S^{k}_{x_i})\setminus N_{T_{b, i}} (x_{i}))\subset R_i\cup \{z'_i\}$ and $\phi_i(N_{T_{b, i}} (x_{i}))\subset X_i$.
By \cref{lem:cycle-is-covered-by-tree}, we can greedily embed the subtrees $T_{b, i}^{(k)}(x_{i, i'})$ to $C(q_i)$ ensuring that $x_{i, i'}$ is embedded into $X_i$ for $i' \in [t_i]$. By removing the used vertices in $R_i$, we obtain $R_{i+1}$ and clearly $|R_{i+1}| \geq |R_{i}| - k\Delta$.

Finally, assume that $2 \leq d \leq k-1$.
By \cref{lem:cycle-is-covered-by-tree}, for each $T_{b,i}$ with $i \in [r']$ and a $q_i$-blow-up $C(q_i)$ of loose cycle of length $\ell$, 
there is an embedding $\psi_i$: $T^{(k)}_{b,i} \to C(q_i)$. 
Note that $T_0^{(k)}$ and $T_b^{(k)}$ are connected by at most $\Delta$ many $k$-edges.
For clarity, we demonstrate that the number of vertices available for connection is sufficient in this case. The number of vertices needed for connection is at most 
\[
  (k - 2) (r' + \Delta) \leq (k - 2) \times \left(\frac{(1-\alpha)n}{(k-1)\eps'm-k} + \Delta\right) \leq O \left( \frac{n} {(\log n)^c} \right). 
\] 
Note that in this case for any two vertices $x, y \in V(H)$, we have 
\[
   \deg_{H}(\{x, y\}, R) \geq \frac{\delta (\delta n)^{k-2}} {2} > (k - 2) (r' + \Delta) (\delta n)^{k - 3}.
\]
Thus, we can choose distinct vertices in $R$ to connect these $k$-expansion subtrees.
Therefore, we get an embedding of $T^{(k)}$ in $H$ such that $v$ is embedded to $w$. 
\end{proof}

\section{Embedding spanning expansion hypertrees} \label{section:proof-main}
Now we shall give the proof of \cref{theorem:main}.
We first recall some definitions.
Let $H$ be a $k$-graph.
A \emph{diamond} in $H$ is a subgraph formed by two edges that intersect in $k-1$ vertices.
Given $w, v \in V(H)$, a diamond $D$ is a \emph{$(w, v)$-diamond} if $w, v$ are the two vertices of degree one in $D$.
A \emph{half-diamond} refers to either of the two edges in $D$ containing $w$ or $v$. We denote the half-diamond containing $v$ as $f^v$ and the one containing $w$ as $f^w$.

\begin{definition}
Let $H$ be a $k$-graph and $(w, w_1,\ldots, w_{k-1})$ be an ordered $k$-set of $V(H)$. 
An {\it absorbing tuple} for $(w, w_1,\ldots, w_{k-1})$ is a tuple $(f^v_1, \dotsc, f^v_{k-1})$, where $f^v_i$ is a half-diamond in $(w_i, v_i)$-diamond containing $v_i$ for each $i \in [k-1]$.
Moreover, those halves are vertex-disjoint; and in addition, we have $w v_1 \ldots v_{k-1} \in E(H)$.

Let $A(w, w_1, \ldots, w_{k-1})$ denote the set of all absorbing tuples for $(w, w_1,\ldots, w_{k-1})$ and let $A(H)$ denote the union of $A(w, w_1,\ldots, w_{k-1})$ for all $k$-sets $(w, w_1,\ldots, w_{k-1})$ of distinct vertices of $V(H)$.
\end{definition}

\begin{definition}
Let $T$ be a tree, $H$ be a $k$-graph, and $\phi$ be an embedding of $T^{(k)}$ in $H$. Let $f^v$ be the half-diamond in $(w, v)$-diamond.
Say that $f^v$ is \emph{immersed in $\phi(T^{(k)})$} if there exists an edge $xy \in E(T)$ such that its $k$-expansion $e \in E(T^{(k)})$ is mapped to $f^v$ via $\phi$, and $v \notin \{ \phi(x), \phi(y)\}$.
We say that an absorbing tuple $(f^v_{1}, \ldots, f^v_{k-1})$ is \emph{immersed in $\phi(T^{(k)})$} if all $f^v_{1}, \ldots, f^v_{k-1}$ are immersed in $\phi(T^{(k)})$.
\end{definition}

We first show that there exists a small set of vertex-disjoint absorbing tuples that contains many absorbing tuples for any $k$-set of $V(H)$.

\begin{lemma}\label{lem:counting-absorb}
Suppose $1/\Delta,1/k, \eps\gg \beta\gg \rho\gg 1/n$.
Let $H$ be a $k$-graph on $n$ vertices with
\[\delta_d(H) \geq \left({1}/{2}  + \eps \right) \binom{n-d}{k-d}.\]
Then there exists a set $\mathcal{A}$ of at most $\beta n$ pairwise vertex-disjoint absorbing tuples such that for every $(w, w_1,\ldots, w_{k-1})$ of $V(H)$, we have 
\[
|A(w, w_1, \ldots, w_{k-1}) \cap \mathcal{A}| \geq \rho n.
\]
\end{lemma}

\begin{proof}
By \cref{degree}, we have $\delta_1(H) \geq (1/2 + \eps) \binom{n-1}{k-1}$.
Therefore, for any distinct vertices $u, v \in V(H)$, we have $|N_{H}(u)\cap N_{H}(v)|\geq 2\eps\binom{n-1}{k-1}$.
Fix a $k$-set of distinct vertices $(w, w_1, \ldots, w_{k-1})$.
Since $\eps\gg 1/n$, there are at least $(1/2+3\eps/4)\binom{n-1}{k-1}$ many edges $w v_1 \ldots v_{k-1}$ disjoint from $\{w_1, \dotsc, w_{k-1}\}$.
For each such $k$-edge, we have that $|N_{H}(w_j)\cap N_{H}(v_j)|\geq 2\eps \binom {n-1}{k-1}$ for each $j\in [k-1]$. Thus for each $v_j$, there are at least 
$\eps \binom{n-1}{k-1}$ many choices for $f^v_{j}$, 
even if the previous half-diamonds $f^v_{j'}$ have been fixed. 
In summary, we have 
\[
|A(w, w_1, \ldots, w_{k-1})|\geq \left(\frac{1}{2} + \frac{3\eps}{4}\right)\binom{n-1}{k-1} \left(\eps\binom{n-1}{k-1}\right)^{k-1} \geq \beta n^{k(k-1)}.
\] 

Let $\mathcal{A'}$ be a random subset of $A(H)$ with each member chosen independently at random with probability $p=cn^{-k(k-1) + 1}$, where $c$ is a constant satisfying $\rho \ll c \ll \beta$.
Note that $|A(H)|\leq n^{k(k-1)}$.
Using Chernoff’s inequality, with probability more than $3/4$ we get that $|\mathcal{A'}|\leq \beta n$.
Similarly, using Chernoff’s inequality and a union bound, with probability more than $3/4$ we get that for each $(w, w_1, \ldots, w_{k-1})$, we have $|A(w, w_1, \ldots, w_{k-1})\cap \mathcal{A'}|\geq c\beta n/2$.
Let $Y$ be the number of pairs of absorbing tuples in $\mathcal{A'}$ that intersect in at least one vertex. Then 
\[
\mathbb E [Y]\leq p^2 n^{2(k-1)^2+2k-3}2^{2(k-1)^2+2k-2}\leq 4^{k^2-k+1} c^2 n.
\]
By Markov's inequality, with probability at least $1/2$, we have $Y\leq 4^{k^2-k+3/2} c^2 n$.
Fix an outcome of $\mathcal{A'}$ such that all above events hold.
For each intersecting pair in $\mathcal{A'}$, remove one of its participating tuples.
By doing so, we get a set $\mathcal{A}$ of size at most $\beta n$, consisting of vertex-disjoint absorbing tuples, and
\[
|A(w, w_1,\ldots, w_{k-1})\cap\mathcal{A}|\geq c\beta n/2- 4^{k^2-k+3/2} c^2 n \geq \rho n,
\]
for each $(w, w_1, \ldots, w_{k-1})$, as desired. 
\end{proof}

We need that absorbing tuples in $\mathcal{A}$ can be immersed by the embedding of any small $k$-expansion hypertree.
The next statement is what we will need for now; its proof will follow from a more general version which we state and prove in \cref{section:immersion}.

\begin{lemma}\label{lem:cover-absorb}
Suppose $1/\Delta, 1/k, \eps\gg \nu\gg \beta\gg 1/n$, and $1 \leq d < k$.
Let $H$ be a $k$-graph on $n$ vertices with  
\[
\delta_d(H)\geq \left(\frac{1}{2} +\eps \right)\binom{n-d}{k-d}
\] and $T_0^{(k)}$ be a $k$-expansion hypertree on $\nu n$ vertices with $\Delta_{1}(T_0^{(k)})\leq\Delta$.
Suppose that $\mathcal{A}$ is a set of at most $\beta n$ pairwise vertex-disjoint absorbing tuples in $H$.
Then there is an embedding $\phi:T_0^{(k)}\to H$ such that every member in $\mathcal{A}$ is immersed in $\phi(T_0^{(k)})$.
\end{lemma}

Using \cref{lem:counting-absorb} and \cref{lem:cover-absorb}, we obtain the following result.

\begin{lemma}[Absorbing Lemma]\label{absorb-lem}
Suppose $1/\Delta,1/k\gg \alpha\gg \delta\gg 1/n$.
Let $T$ be a tree with $\Delta_1(T) \leq \Delta$, such that its $k$-expansion has exactly $n$ vertices.
Let $T_1 \subseteq T$ be a subtree on at least $(1 - \alpha)|V(T)|$ vertices.
Let $H$ be an $n$-vertex $k$-graph and 
$\psi$ be an embedding $\psi: \smash{T_1^{(k)}} \to H$.
Suppose that $\mathcal{A}$ is a family of pairwise vertex-disjoint absorbing tuples in $H$ such that
\stepcounter{propcounter}
\begin{enumerate}[label = {{\rm (\Alph{propcounter}\arabic{enumi})}}]
    \item every member in $\mathcal{A}$ is immersed in $\psi(\smash{T_1^{(k)}})$; and
    \item for any $k$-set $(w, w_1,\ldots, w_{k-1})$ of $H$, there are at least $\rho n$ absorbing tuples in $\mathcal{A}$.
\end{enumerate}
Then there exists an embedding of $\smash{T^{(k)}}$ in $H$.
\end{lemma}

\begin{proof}
Let $n_1 = |V(\smash{T_1^{(k)}})|$, note that $n_1 \geq (1-2\alpha)n$ holds.
Let $V(H) \setminus V(\psi(\smash{T_1^{(k)}})) = \{x_{1}, x_{2}, \ldots, x_{n-n_1}\}$ and $E(\smash{T_1^{(k)}})=\{e_1, e_2, \dots, e_{t_1}\}$.
Recall that every $k$-expansion hypertree has a number of vertices that is $1 \bmod k-1$, and therefore $n - n_1$ is divisible by $k-1$.
Let $t:=\frac{n-n_1}{k-1}$.
Note that since $T_1 \subseteq T$ is a subtree, we can obtain $T$ from $T_1$ by iteratively adding $t$ leaves to $T_1$.
Thus,  there is a sequence of $k$-expansion hypertrees $\smash{T^{(k)}_0}, \smash{T^{(k)}_1}, \dotsc, \smash{T^{(k)}_{t}}$, such that $\smash{T^{(k)}_{t}} = T^{(k)}$, and for each $i \in [t]$, we have $E(\smash{T^{(k)}_i}) = \{e_1, e_2, \dotsc, e_{t_1 + i}\}$, and $\smash{T^{(k)}_i}$ is obtained from $\smash{T^{(k)}_{i-1}}$ by adding the edge $e_{t_1+i-1}$.

Inductively, we will build for each $i \in \{0\} \cup [t]$ an embedding $\phi_{i}: \smash{T^{(k)}_{i}} \to H$ along with a subset $\mathcal{A}_{i} \subseteq \mathcal{A}$ satisfying
\begin{enumerate}
    \item \label{item:embeddingstandard-i} $V(\phi_{i}(\smash{T^{(k)}_{i}}))= V(\psi(\smash{T^{(k)}_0})) \cup \{x_{1}, \ldots, x_{(k-1)i}\};$
    \item $|\mathcal{A}_{i}| \leq i(k-1)!$; and
    \item \label{item:embeddingstandard-iii} every absorbing tuple in $\mathcal{A} \setminus \mathcal{A}_{i}$ is immersed in $\phi_{i}(\smash{T_i^{(k)}})$.
\end{enumerate}
In this way, the final embedding $\phi_{t}$ will be the desired embedding of $T^{(k)}$ in $H$.

The base case $i=0$ holds trivially by letting $\phi_0 := \psi$.
Now suppose that for some $i$ with $0 \leq i <t$, we have constructed $\phi_{i}$ and $\mathcal{A}_{i}$ satisfying \ref{item:embeddingstandard-i}--\ref{item:embeddingstandard-iii}, and we aim to define $\phi_{i+1}$ and $\mathcal{A}_{i+1}$.
 Let $e_{t_1+i+1}= \{z_{1}, \dots, z_{k}\}$ be the next edge to embed, where $z_{1} \in V(\smash{T^{(k)}_{i}})$ and $z_{2},\dots, z_{k} \notin V(T^{(k)}_{i})$. 
 Let $w_{1} := \phi_{i}(z_{1})$ and $w_{j} := x_{(k-1)i+j-1}$ for $j\in \{2,3,\dots,k\}$.
 Since 
$\left| A(w, w_{1},\dots, w_{k-1}) \cap (\mathcal{A} \setminus \mathcal{A}_{i}) \right| \geq \rho n - i(k-1)! \geq (\rho - 2(k-1)! \alpha)n > 0,$
there exists an absorbing tuple $(f^v_{1},\dots, f^v_{k-1}) \in \mathcal{A} \setminus \mathcal{A}_{i}$ for $(w, w_{1},\dots, w_{k-1})$.
For $j \in [k-1]$, write $f^v_{j} = (v_{j}, v_{j}^1, \ldots, v_{j}^{k-1}).$
Update $\mathcal{A}_{i+1} := \mathcal{A}_{i} \cup \mathcal{S} $, where
$ \mathcal{S}$ consists of all ordered $(k - 1)$-tuples formed by permutations of the set $\{f^v_{1}, \dots, f^v_{k-1}\}$.
Now we define $\phi_{i+1}$ from $\phi_{i}$ via
\[\phi_{i+1}(v) = 
\begin{cases} 
w_{j} & \text{if } v = \phi_{i}^{-1}(v_{j}) \text{ for } j \in [k-1], \\
v_{j} & \text{if } v = z_{j} \text{ for } j \in [k-1], \\
\phi_{i}(v) & \text{otherwise}.
\end{cases}\]
By construction, $\phi_{i+1}$ is an embedding of $\smash{T^{(k)}_{i+1}}$ satisfying \ref{item:embeddingstandard-i}--\ref{item:embeddingstandard-iii}.
This finishes the induction and the proof.
\end{proof}

\begin{proof}[Proof of \cref{theorem:main}]
Let $k, d, \eps, \Delta$ be given, so that $k > d \geq 1$, $\eps>0$ and $\Delta\in \mathbb N$.
Let $H$ be an $n$-vertex $k$-graph with 
\[\delta_d(H) \geq \left(\max\left\{\frac{1}{2}, \deltaLHC{k}{d}\right\} + \eps\right)\binom{n-d}{k - d},\] 
and let $T$ be an $n'$-vertex tree with
$\Delta_1(T) \le \Delta$, where $n=(k-1)n'-k+2$.
We need to find an embedding of $T^{(k)}$ in $H$.
Clearly, we have $|V(T^{(k)})|=n$ and $\Delta_1(T^{(k)}) \le \Delta$.
For the remainder of the proof, we choose $\nu, \beta, \rho, \alpha$ so that
\begin{equation}\label{equa-2}
1/k, 1/\Delta, \eps\gg \nu \gg \beta \gg \rho\gg \alpha\gg 1/n,  
\end{equation}
and furthermore $\ell$ is an odd integer.

First, we find the absorbing structures.
Using \cref{lem:counting-absorb} to $H$, we get a family $\mathcal{A}$ consisting of at most $\beta n$ pairwise vertex-disjoint absorbing tuples such that for every $k$-set $(w, w_1,\ldots, w_{k-1})$ in $H$, we have 
\[
|A(w, w_1,\ldots, w_{k-1})\cap\mathcal{A}|\geq \rho n.
\]

Now, fix an arbitrary leaf vertex $x \in V(T)$, and form a rooted tree $(T, x)$ by rooting $T$ at $x$.  
Let $u$ be a vertex in $T$ such that 
$\nu  n' \le |V(T(u))| \le \Delta \nu  n'$
(such a vertex can be found by letting $u$ be a maximal-depth vertex with $\nu n' \le |V(T(u))|$).
Note that the $k$-expansion hypertree $T^{(k)}(u)$ has size at most $\Delta \nu n$. 
Applying \cref{lem:cover-absorb},
there is an embedding
$\phi_0: T^{(k)}(u) \to H$ 
such that every element in $\mathcal{A}$ is immersed in $\phi_0(T^{(k)}(u))$. Let $u':=\phi_0 (u)$.
Let $T_1 := T - (V(T(u)) \setminus \{u\})$, $H_1 := H - (V(\phi_0 (T^{(k)}(u))) \setminus \{u'\})$ and $|V(H_1)| := n_1$. Note that
\[\delta_d(H_1) \geq \left(\max\left\{\frac{1}{2}, \deltaLHC{k}{d}\right\} + \frac{\eps}{2} \right)\binom{n_1 - d}{k - d}.\]

Secondly, we find an embedding of the almost spanning $k$-expansion hypertree $T_1^{(k)}$ in $H_1$.
Let $T_2$ be obtained from $T_1$ by removing leaves, different from $u$, repeatedly, in such a way that $|V(T^{(k)}_1)| - |V(T^{(k)}_2)|=\alpha n$, where $\alpha$ is a constant satisfying (\ref{equa-2}).
We invoke \cref{almost-cover} with $T^{(k)} := T_2^{(k)}$, $H := H_1$, $v := u$ and $w := u'$, there is an embedding $\phi_1: T_2^{(k)} \to H_1$ such that $\phi_1(u) = u'$.
Let $T^{(k)}_0:=T^{(k)}_2\cup T^{(k)}(u)$.
Thus, we obtain an embedding of $T_0^{(k)}$ in $H$.

Finally, we use the absorbers. By (\ref{equa-2}),
$|V(H)|-|V(T_0^{(k)})|=\alpha n \ll \rho n$.  
Applying \cref{absorb-lem}, there is an embedding of $T^{(k)}$ in $H$.
\end{proof}

\section{Basic gadgets} \label{section:basicgadgets}

In this section, we describe how to find the structures for the basic switches in the host $k$-graph $H$.
In the next section, we will combine the basic switches to form complex gadgets.
Next, we will describe how to embed a small tree to `immerse' in these gadgets correctly.

\subsection{The diamond graph} \label{ssection:diamond}
Given an $n$-vertex $k$-graph $H$ and a constant $\gamma > 0$, we define the  \emph{$\gamma$-diamond graph} $G$ of $H$ with $V(G) = V(H)$ such that $xy \in E(G)$ if and only if there are at least $\gamma \binom{n}{k-1}$ many $(x,y)$-diamonds in $H$. 

\begin{lemma} \label{lemma:diamondgraph-mindeg}
Let $0 < \gamma \leq 1/100$ and $n$ be sufficiently large.
Let $H$ be an $n$-vertex $k$-graph with $\delta_{k-1}(H) \geq n/3$, and let $G$ be its $\gamma$-diamond graph.
Then $\delta(G) \geq n/10$.
\end{lemma}

\begin{proof}
For a set $S \subseteq V(H)$, we denote by $L_H(S)$ the \emph{link graph of $S$ in $H$}, which is the $(k-|S|)$-graph on the vertex set $V(H)\setminus S$ with an edge $e$ whenever $e\cup S \in E(H)$.
Let $x$ be any vertex of $H$, and let $S \subseteq V(H) \setminus \{x\}$ be any $(k-2)$-set.
The link graph $L_H(S)$ is a graph, and by \cref{degree} it has minimum degree at least $n/3$.
Hence, the number of paths $x u v$ in $L_H(S)$ is at least $(n/3)(n/3 - 1)$.
Note that each such path yields an $(x,v)$-diamond.
The number of choices for $(S, u, v)$ as before is at least \[
\frac n3 \times \left(\frac n3 - 1\right)\times \binom{n-1}{k-2} \geq \frac{(k-1) (n-3)}{9} \binom{n}{k-1}.\]
Let $U$ be the set of vertices $v$ that participate in at least $\gamma (k-1) \binom{n}{k-1}$ tuples $(S, u, v)$ as before.
Any vertex can participate in at most $\binom{n}{k-2}n = (k-1)\binom{n}{k-1} + o(n^{k-1})$ such tuples.
Then
\[ \frac{(k-1) (n-3)}{9} \binom{n}{k-1} \leq |U| (k-1)\binom{n}{k-1} + (n - |U|) (k-1)\gamma \binom{n}{k-1} + o(n^{k-1}), \]
which implies $|U| \geq n(1/9 - \gamma + o(1))/(1 + \gamma) \geq n/10$.
Every $(x,v)$-diamond yields at most $k-1$ tuples $(S,u,v)$, so we have $U \subseteq N_G(x)$, and we are done.
\end{proof}

\begin{lemma} \label{lemma:diamondgraph-indep}
Given $\eps \geq \gamma > 0$, let $H$ be an $n$-vertex $k$-graph with $\delta_{k-1}(H) \geq (1/3 + \eps)n$, and let $G$ be its $\gamma$-diamond graph.
Then $\alpha (G)\le 2$. 
\end{lemma}

\begin{proof}
 Given any three vertices $x, y, z \in V(H)$,
    we have \[
    \deg_H(x) + \deg_H(y) + \deg_H(z) \geq (1 + 3 \eps) \binom{n}{k-1}.\]
    Thus, for at least one pair of vertices, say $\{x,y\}$, we have \[|N_H(x) \cap N_H(y)| \geq \eps \binom{n}{k-1} \geq \gamma \binom{n}{k-1}.\]
    This means that $xy \in E(G)$, so $\{x,y,z\}$ is not independent in~$G$.
\end{proof}

We want to understand the components of $G$.
In fact, we require a bit more and want to work with components that are `inseparable', meaning that they have a constant proportion of edges across every possible cut.
The main property we want to use from inseparable graphs is that they have many bounded-length paths between any pair of vertices (see~\cref{lemma:manypaths}).

Given a graph $G = (V,E)$, a vertex partition $\{U_1, U_2\}$ with $U_1, U_2$ nonempty is a \emph{$\mu$-separation} if $e(U_1, U_2) < \mu|U_1||U_2|$. 
We say that a graph $G$ is \emph{$\mu$-inseparable} if it does not have a $\mu$-separation.

\begin{lemma} \label{lemma:diamonggraph-separation}
Given $\mu\gg 1/n$, let $G$ be an $n$-vertex graph with $\delta(G) \geq 7 \mu n$ and $\alpha(G) \leq 2$.
	Then $G$ satisfies either
    \begin{enumerate}
        \item $G$ is $\mu$-inseparable, or
        \item there is a $\mu$-separation $\{U_1,U_2\}$ of  
  $V(G)$ such that $G[U_1]$ and $G[U_2]$ are $\mu$-inseparable.
    \end{enumerate}
\end{lemma}

\begin{proof}
    Suppose $G$ is not $\mu$-inseparable.
     Among all $\mu$-separations $\{U_1, U_2\}$ of $V(G)$, choose one that minimizes $e(U_1, U_2)/|U_1||U_2|$.
    We now show that $G[U_1]$ and $G[U_2]$ are $\mu$-inseparable.
    
    Since $e(U_1, U_2) < \mu |U_1||U_2|$, there exists $u \in U_2$ with $\deg_G(u, U_1) < \mu |U_1|$.
    For any pair $x, y \in U_1 \setminus N(u)$, we must have $xy \in E(G)$, since $\alpha(G) \leq 2$.
    Hence, there exists a clique on vertex set $K_1$ of size at least $(1 - \mu)|U_1|$ in $G[U_1]$.
    Similarly, there exists a clique on vertex set $K_2$ of size at least $(1 - \mu)|U_2|$ in $G[U_2]$.
    Let $V_i = U_i \setminus K_i$ for $1 \leq i \leq 2$.

    We claim that for every $v \in V_1$, $\deg_G(v, K_1) > 2 \mu n$.
    Suppose otherwise.
    Hence, there exists a vertex $v \in V_1$ such that $\deg_G(v, U_1) = \deg_G(v, K_1) + \deg_G(v, V_1) \leq 2 \mu n + |V_1| \leq 2\mu n + \mu|U_1| \leq 3 \mu n$.
    Therefore, $\deg_G(v, U_2) \geq \delta(G) - 3 \mu n \geq 4 \mu n \geq \deg_G(v, U_1) + \mu n$.
    Now consider the partition $\{U'_1, U'_2\}$, where $U'_1 = U_1 \setminus \{v\}$ and $U_2' = U_2 \cup \{v\}$.
    We have $e(U'_1, U'_2) = e(U_1,U_2) - \deg_G(v,U_2) + \deg_G(v,U_1) \leq e(U_1,U_2) - \mu n$,
    which implies \[\frac{e(U'_1,U'_2)}{|U'_1||U'_2|} = \frac{e(U'_1,U'_2)}{|U_1||U_2| + |U_1| - |U_2| - 1} \leq \frac{e(U_1,U_2) - \mu n}{|U_1||U_2| + |U_1| - |U_2| - 1} < \frac{e(U_1,U_2)}{|U_1||U_2|}, \]
    where the last inequality follows from $|U_2|-|U_1| + 1 < n$ and $e(U_1,U_2) < \mu|U_1||U_2|$.
    This contradicts the minimality of $\{U_1, U_2\}$ and proves the claim.
    Similarly, for every $v \in V_2$, $\deg_G(v, K_2) > 2 \mu n$.
    In particular, note that this implies that $\min\{ |K_1|, |K_2| \} > 2 \mu n$.

    Now we show that $G[U_1]$ is $\mu$-inseparable.
    Let $\{X_1, X_2\}$ be a partition of $U_1$, with $|X_1| \leq |X_2|$.
    Suppose first that $|X_1 \cap K_1| \leq \mu n$.
    Note that every $v \in U_1$ has at least $2 \mu n$ neighbors in $K_1$ (we proved this for $v \in V_1$, and for $v \in K_1$ it is obvious, because it is a clique on more than $2 \mu n$ vertices).
    Hence, we have that $e(X_1, X_2) \geq |X_1|( 2 \mu n - |X_1 \cap K_1|) \geq |X_1| \mu n \geq \mu |X_1||X_2|$.
    Next, suppose that $|X_1 \cap K_1| > \mu n$.
    Since $|V_1| \leq \mu |U_1|$, this implies that $|X_1 \cap K_1| > |X_1|/2$.
    Since $|X_1| \leq |X_2|$, we have $|X_2| \geq |U_1|/2 \geq 2 \mu |U_1| \geq 2 |V_1|$, so we have $|X_2 \cap K_1| \geq |X_2|/2$.
    Since $K_1$ is a clique, we have $e(X_1, X_2) \geq |X_1 \cap K_1||X_2 \cap K_1| \geq \frac{1}{4}|X_1||X_2| \geq \mu |X_1||X_2|$.
    Hence $G[U_1]$ is $\mu$-inseparable.
    The proof that $G[U_2]$ is $\mu$-inseparable is analogous.
\end{proof}

\subsection{Proto-balancers}

Given a $k$-graph $H$ and a vertex partition $\{A,B\}$ of $V(H)$, an \emph{$(A,B)$-diamond} is an $(x,y)$-diamond for some $x \in A$ and $y \in B$.
We begin our analysis by studying the structure of $H$ assuming that it has a vertex partition without any $(A, B)$-diamonds.
Later, we will transfer this information to an imperfect case, where there are few $(A, B)$-diamonds: this will correspond to the case where there is a $\mu$-separation in the $\gamma$-diamond graph.

The structure we will look at is that of \emph{proto-balancers}, which we define now.
An \emph{$(A, B)$-proto-balancer} in $H$ is a subgraph defined as follows.
There is a $(k-3)$-set $S \subseteq V(H)$ and three $3$-sets $X_1, X_2, X_3 \subseteq V(H) \setminus S$ such that
\begin{enumerate}
    \item $S \cup X_i$ is an edge in $H$, for all $i \in [3]$,
    \item $|X_1 \cap A| = |X_2 \cap A| = |X_1 \cap X_2| = |X_1 \cap X_2 \cap A| = 2$,
    \item $|X_3 \cap A| = 0$,
    \item $|X_1 \cap X_3| = |X_2 \cap X_3| = 1$.
\end{enumerate}
We say that $S$ is the \emph{base} of the $(A, B)$-proto-balancer.
A $(B, A)$-proto-balancer is the same but exchanging the roles of $A$ and $B$; and we just say \emph{proto-balancer} if it is an $(A, B)$-proto-balancer or a $(B, A)$-proto-balancer.
Note that proto-balancers have $k+2$ vertices.
The key property of $(A,B)$-proto-balancers is that, if $e_i = S \cup X_i$ for $i \in [3]$ are the edges of $H$, we have that $|e_1 \cap A| = |e_2 \cap A| = |e_3 \cap A| + 2$. 

Now we show that $k$-graphs with no $(A, B)$-diamonds have proto-balancers.

\begin{lemma}[One proto-balancer without diamonds] \label{lemma:nodiamonds-oneproto}
Let $k \geq 3$ and $n \geq 2k-4$, and let $H$ be an $n$-vertex $k$-graph with $\delta_{k-1}(H) > n/3$, and a vertex partition $\{A,B\}$ with $|A|, |B| > 0$.
Suppose that there are no $(A, B)$-diamonds in $H$.
Then $H$ has a proto-balancer.
\end{lemma}

In the proof we shall repeatedly use the following observation.
If there are no $(A,B)$-diamonds, then for every $(k-1)$-set $S \subseteq V(H)$ we have that either $N_H(S) \subseteq A$ or $N_H(S) \subseteq B$.
(Otherwise we could find $a \in A, b \in B$ such that $S \cup \{a\}$ and $S \cup \{b\}$ are edges in $H$, but this is an $(a,b)$-diamond.)

\begin{proof}[Proof of \cref{lemma:nodiamonds-oneproto}]
Suppose first that $k=3$.
Let $a \in A$, $b \in B$.
By the observation, we have $N_H(ab) \subseteq A$ or $N_H(ab) \subseteq B$.
Without loss of generality (otherwise we swap $A$ and $B$), we suppose the former holds.
Now, let $a' \in N_H(ab) \cap A$.
Since $b \in N_H(aa')$, again by the observation, we must have $N_H(aa') \subseteq B$.
Hence, $|B| > n/3$ and $|A| < 2n/3$.

    Now, let $b' \in N_H(aa') \cap B$ be distinct from $b$.
    Note that $N_H(bb')\cap (N_H(ab) \cup \{a\})=\emptyset$, otherwise it would create a $(A,B)$-diamond.
    Since $|A| < 2n/3$, $|N_H(ab)|, |N_H(bb')| > n/3$ and $N_H(ab') \subseteq A$, then we cannot have $N_H(bb') \subseteq A$.
    By the observation, then it must happen that $N_H(bb') \subseteq B$.
    Let $b'' \in N_H(bb') \cap B$.
    Then we found an $(A,B)$-proto-balancer with $S = \emptyset$, $X_1 = \{a, a', b\}$, $X_2 =\{a,a', b'\}$, $X_3 = \{b, b', b''\}$.

    Finally, if $k > 3$, let $S \subseteq V(H)$ be an arbitrary $(k-3)$-set such that $A \setminus S \neq \emptyset$ and $B \setminus S \neq \emptyset$, and apply the argument above in the $3$-uniform link graph $L_H(S)$.
    This gives an $(A,B)$-proto-balancer with base $S$.
\end{proof}

We leverage the above result to find many proto-balancers if there are a few diamonds.

\begin{lemma}[Many proto-balancers with few diamonds] \label{lemma:fewdiamond-manyproto}
    Let $1/k, \eps \gg \gamma \geq \mu \gg 1/n$ and $\eps \gg \rho$.
    Let $H$ be an $n$-vertex $k$-graph with $\delta_{k-1}(H) \geq (1/3 + \eps)n$, and let $G$ be its $\gamma$-diamond graph.
    Suppose $G$ has a $\mu$-separation $\{A, B\}$.
    Then there are at least $\rho n^{k+2}$ proto-balancers in $H$.
\end{lemma}

To prove this, we will use the following concentration inequality~\cite[Corollary 2.2]{Greenhill-Isaev-Kwan-McKay-2017}.

\begin{lemma} \label{lemma:GIKM}
    Let $X$ be a set of size $n$, and let $q < n/2$.
    Let $f$ be a function over the sets of size $q$ of $X$ such that $|f(S) - f(S')| \leq z$ if $|S \cap S'| = q-1$.
    If $S$ is a $q$-subset of $X$ chosen uniformly at random, then $\probability[ |f(S) - \expectation[f(S)] | \geq t ] \leq 2e^{- 2 t^2 / ( q z^2)}$.
\end{lemma}

\begin{proof}[Proof of \cref{lemma:fewdiamond-manyproto}]
    Let $q$ be a large integer so that $1/k, \eps \gg 1/q \gg \gamma, \rho$.
    Let $S$ be a $q$-set of vertices drawn uniformly at random.
    We gather some probabilistic estimates.

    \begin{enumerate}
        \item $\probability[ |A \cap S| = 0 ]\leq 1/8$ and $\probability[ |B \cap S| = 0 ] \leq 1/8$.
    \end{enumerate}

    We prove the statement for $A$, and the proof of $B$ is identical.
    By \cref{lemma:diamondgraph-mindeg} we have that $\delta(G) \geq n/10$.
    Since $\{A, B\}$ is a $\mu$-separation, there must exist $a \in A$ with less than $\mu |B| \leq \mu n$ neighbors in $B$, hence $|A| \geq \deg_G(a, A) \geq \delta(G) - \mu n \geq n/20$.
   Note that $|S \cap A|$ is hypergeometric with $\expectation[|S \cap A|] = q|A|/n \geq q/20$.
    Changing one vertex changes $|S \cap A|$ by at most one.
    Hence, we can apply \cref{lemma:GIKM} with $z = 1$, $t = q/20$ to obtain $\probability[|S \cap A| = 0] \leq 2 \exp( - q / 200) < 1/8$, where the last inequality follows from $1 \gg 1/q$.

    \begin{enumerate}[resume]
        \item $\probability[\:\text{$H[S]$ contains an $(A,B)$-diamond }] < 1/8$.
    \end{enumerate}

    Every pair $ab \notin E(G)$ with $a \in A$, $b \in B$ yields at most $\gamma \binom{n}{k-1}$ many $(a,b)$-diamonds, and the pairs in $E(G)$ can yield at most $\binom{n}{k-1}$ many $(a,b)$-diamonds.
    Hence, the total number of $(A,B)$-diamonds is at most $\gamma \binom{n}{k-1} |A||B| + \binom{n}{k-1} \mu |A||B| \leq 2 \gamma \binom{n}{k-1} n^2/4 \leq (k+1)^2 \gamma \binom{n}{k+1}$.
    Hence, the number $X$ of such diamonds in $S$ has expected value at most $(k+1)^2 \gamma \binom{q}{k-1} < 1/8$, where we use $1/k, 1/q \gg \gamma$ in the last inequality.
    By Markov's inequality, $\probability[X \geq 1] \leq \expectation[X] < 1/8$, as required.

    \begin{enumerate}[resume]
        \item $\probability[ \:\delta_{k-1}(H[S]) \leq  q/3 \: ] \leq \binom{q}{k-1} 2 e^{- \eps^2 q / 2} < 1/8$.
    \end{enumerate}

    The first inequality also follows from \cref{lemma:GIKM}, together with a union bound (see, e.g., \cite[Lemma 3.4]{Ferber-Kwan-2022} for a similar argument).
    The second inequality follows from the choice $1/k, \eps \gg 1/q$.

    Putting together the three facts above, we have that there are at least $\binom{n}{q}/2$ many $q$-sets $S$ such that $|A \cap S|, |B \cap S| > 0$, $H[S]$ has no $(A,B)$-diamond, and $\delta_{k-1}(H[S]) > q/3$.
    For each of these sets, \cref{lemma:nodiamonds-oneproto} implies that $H[S]$ contains a proto-balancer.
    Since each proto-balancer has $k+2$ vertices, each proto-balancer obtained in this way can appear in at most $\binom{n-k-2}{q-k-2}$ such sets $S$.
    Hence, the number of distinct proto-balancers is at least 
    \[ \frac{\frac{1}{2} \binom{n}{q}}{\binom{n-k-2}{q-k-2}} = \frac{1}{2 \binom{q}{k+2}} \binom{n}{k+2} \geq \rho n^{k+2}, \]
    where the last inequality is by the choice $1/k, 1/q \gg \rho$.
\end{proof}

\subsection{Balancers} \label{ssection:balancers}
We now describe \emph{balancers}, which are auxiliary gadgets found in a host $k$-graph $H$ chosen with respect to a partition $\{A,B\}$ of $V(H)$.
For convenience, we call the $k$-expansion of $K_{2, t}$ an \emph{uncolored $t$-balancer}, denoted by $J$, with anchor vertex classes $\{x_1, x_2\}$ and $\{y_1, \dotsc, y_t\}$, and for each $\ell \in [2]$ and $i \in [t]$,  let $f_{\ell, i} \in E(J)$ be the edge containing $x_{\ell}$ and $y_i$.
For each $\ell \in [2]$, 
let $J_\ell\subseteq J$ be the subgraph consisting of the edges $f_{\ell,i}$ for all $i \in [t]$. Note that both $J_1$ and $J_2$ are $t$-stars and $|V(J)\setminus V(J_1)|=|V(J)\setminus V(J_2)|$. Hence, we say that $J_1, J_2$ are the \emph{halves} of $J$. 

Now we will define appropriate vertex partitions of uncolored balancers.
We will abuse the notation here and pretend the partition is $\{A,B\}$ for all the relevant $k$-graphs, even if they have different vertex sets, but this should not create confusion.
Given an uncolored $t$-balancer $J$, and positive integers $t_A, t_B$ such that $t_A + t_B = t$, and a vertex partition $\{A,B\}$ of $V(J)$, we will say that $J$ is a \emph{$(t_A, t_B)$-balancer} if, by setting $I_A = [t_A]$ and $I_B = [t] \setminus [t_A]$, we have
\begin{enumerate}
    \item $x_1 \in A$, $x_2 \in B$,
    \item $y_i \in A$ for each $i \in I_A$ and $y_j \in B$ for each $j \in I_B$,
    \item for each $i \in I_A$, we have $|(f_{1,i}\setminus\{x_1,y_i\})\cap B| = |(f_{2,i}\setminus\{x_2,y_i\}) \cap B|+1$, and
    \item for each $j \in I_B$, we have $|(f_{1,j}\setminus\{x_1,y_j\})\cap B| = |(f_{2,j}\setminus\{x_2,y_j\}) \cap B|-1$.
\end{enumerate}
In an embedding, we would use either the edges of $J_1$ or the edges of $J_2$ for a $(t_A, t_B)$-balancer $J$.
Note that we have
\[|V(J_1)\cap B|=t_B+\sum_{i\in I_A}|(f_{1,i}\setminus\{x_1,y_i\})\cap B|+\sum_{j\in I_B}|(f_{1,j}\setminus\{x_1,y_j\})\cap B| \]
and 
\[|V(J_2)\cap B|=1+t_B+\sum_{i\in I_A}(|(f_{1,i}\setminus\{x_1,y_i\})\cap B|-1) +\sum_{j\in I_B}(|(f_{1,j}\setminus\{x_1,y_j\})\cap B|+1).\]
By switching from the edges of $J_1$ to the edges of $J_2$, we precisely gain $1 + t_B - t_A$ vertices in $B$ (this number can be negative, in which case we lose this number of vertices in $B$).
Say that $1 + t_B - t_A$ is the \emph{capacity} of a $(t_A, t_B)$-balancer.
If $t$ is even, by choosing $t_A = t_B = t/2$, we have that a $(t_A, t_B)$-balancer has capacity~$1$.
If $t$ is odd, by choosing $t_A - 1 = t_B = \lceil t/2 \rceil$, we have that a $(t_A, t_B)$-balancer has capacity $2$.

Now we wish to show that if the $\gamma$-diamond graph of $H$ has a separation, then there are many $(t_A, t_B)$-balancers.
For this, we will use the fact that there is a color-preserving homomorphism from any $(t_A, t_B)$-balancer to any $(A, B)$-proto-balancer, i.e., it maps vertices from $A$ to vertices in $A$, and vice versa.
This will allow us to conclude using supersaturation.

\begin{lemma}[Balancers from few diamonds] \label{lemma:balancers}
    Let $k \geq 3$, let $t_A, t_B \geq 1$ be integers, and $1/k, \eps \gg \gamma \geq \mu \gg 1/n$ and $\eps,  1/t_A, 1/t_B\gg \rho$.
    Let $H$ be an $n$-vertex $k$-graph with $\delta_{k-1}(H) \geq (1/3 + \eps)n$, and let $G$ be its $\gamma$-diamond graph.
    Suppose $G$ has a $\mu$-separation $\{A, B\}$.
    Then there are at least $\rho n^{(2k-3)(t_A + t_B)+2}$ many $(t_A, t_B)$-balancers in $H$.
\end{lemma}
 
\begin{proof}
    Let $k, t_A, t_B, \eps$ be given, we choose $\gamma, \mu, \rho', \rho$ so that $1/k, \eps \gg \gamma \geq \mu \gg 1/n$ and $\eps,  1/t_A, 1/t_B \gg\rho'\gg \rho$. Applying~\cref{lemma:fewdiamond-manyproto} with $\rho'$, there are at least $\rho'n^{k+2}$ proto-balancers in $H$. By the pigeonhole principle, either there are at least $\frac{\rho'}{2}n^{k+2}$ $(A, B)$-proto-balancers in $H$ or at least $\frac{\rho'}{2}n^{k+2}$ $(B, A)$-proto-balancers in $H$.  Without loss of generality, assume that there are at least $\frac{\rho'}{2}n^{k+2}$ $(A, B)$-proto-balancers in $H$. 

    Consider any $(A,B)$-proto-balancer $P$ with base $S$ and edge set $\{S\cup \{a_1,a_2,b_1\}, S\cup\{a_1,a_2,b_2\}, S\cup\{b_1,b_2,b_3\}\}$, where $\{a_1, a_2\}\subset A$ and $\{b_1, b_2, b_3\}\subset B$.
    Let $\ell = |S \cap B|$ and note that $\ell \in \{0, \dotsc, k-3\}$. Fix $t=t_A+t_B$.
    Let $J$ be an uncolored $t$-balancer with anchor vertex classes $\{x_1, x_2\}$ and $\{y_1, \dotsc, y_t\}$, and edge set $E(J)=\{f_{1, i}, f_{2, i}\}_{i \in [t]}$. 

    Let $I_A = [t_A]$ and $I_B = [t] \setminus [t_A]$.
    For each $i \in [t]$, select a vertex $u_i \in f_{1, i}\setminus \{x_1, y_i\}$ and $v_i \in f_{2, i}\setminus \{x_2, y_i\}$.
    We first show that there is a homomorphism $\Phi: J \to P$ by setting 
    \[\Phi(v) = 
    \begin{cases} 
    a_1 & \text{ if } v =x_1; \text{ or } v = v_i \text{ for } i \in I_A, \\
    a_2 & \text{ if } v = y_i \text{ for } i \in I_A; \text{ or } v = u_j \text{ for } j \in I_B, \\
    b_1 & \text{ if } v = x_2; \text{ or } v = u_i \text{ for } i \in I_A, \\
    b_2 & \text{ if } v = y_j \text{ for } j \in I_B, \\
    b_3 & \text{ if } v = v_j \text{ for } j \in I_B,
\end{cases}\]
and $\Phi(f_{1,i}\setminus \{x_1, u_i, y_i\})= \Phi(f_{2,i}\setminus \{x_2, v_i, y_i\})=S$ for each $i \in [t]$.
    We define a partition $\{A, B\}$ on $V(J)$ where $x \in A$ if and only if $\Phi(x) \in A$.
    This partition makes $J$ a $(t_A, t_B)$-balancer.
    Indeed, by $\Phi$ we have
      $|(f_{1,i}\setminus\{x_1,y_i\})\cap B| =\ell+1$ and $|(f_{2,i}\setminus\{x_2,y_i\}) \cap B| =\ell$ for $i\in I_A$; for $j\in I_B$, we have $|(f_{1,j}\setminus\{x_1,y_j\})\cap B|=\ell$ and $|(f_{2,j}\setminus\{x_2,y_j\}) \cap B|=\ell+1$. 
Note that in the homomorphism $\Phi$, each vertex $x \in V(P)$ corresponds to at most $2t$ pre-images in $V(J)$.
    Thus, for the $2t$-blow-up $P(2t)$ of $P$, it's easy to see that there is an embedding $\phi: J \to P(2t)$.
    This means that we can find a $(t_A,t_B)$-balancer in a $P(2t)$.

    Now we show that there are at least $\frac{\rho'}{2}n^{2t(k+2)}$ many $P(2t)$ in $H$. Consider a $(k+2)$-graph $F$ on vertex set $V(H)$ and edge set 
    \[
    E(F):=\left\{e\in \binom{V(H)}{k+2}: H[e] \text{ contains an } (A, B) \text{-proto-balancer} \right\}.
    \]  
    Clearly, $|E(F)|\geq \frac{\rho'}{2}n^{k+2}$. By \cref{supersa}, there exists $\rho>0$ such that $F$ contains at least $\rho n^{2t(k+2)}$ copies of $K^{(k+2)}_{k+2}(2t)$, where $K^{(k+2)}_{k+2}(2t)$ denotes the $2t$-blow-up of an edge of $F$. This means that there are at least $\rho n^{2t(k+2)}$ many $P(2t)$ in $H$. Since $2t(k+2)=2(t_A+t_B)(k+2)\geq (2k-3)(t_A+t_B)+2$, we know that there are at least $\rho n^{(2k-3)(t_A + t_B)+2}$ many $(t_A, t_B)$-balancers in $H$. 
\end{proof}

\subsection{Diamond-chains} \label{ssection:vertexgadgets} 
Given vertices $u,v$, a \emph{$(u, v, \ell)$-diamond-chain} is given by a sequence of distinct vertices $x_0, \dotsc, x_\ell$ with $x_0 = u$, $x_\ell = v$, and a collection of $\ell$ diamonds $D_1, \dotsc, D_\ell$ such that for each $1 \leq i \leq \ell$, $D_i$ is an $(x_{i-1}, x_i)$-diamond, and all diamonds are vertex-disjoint except if they are consecutive $D_{i}, D_{i+1}$, where they intersect only in~$x_i$. 
Given a $(u,v, \ell)$-diamond-chain formed by the edge set $\mathcal S=\{e_i, e'_i\}_{i\in [\ell]}$ with $u\in e_1$, $v\in e'_{\ell}$ and $|e_i\cap e'_i|=k-1$ for $i\in [\ell]$, let $\mathcal S_u \subseteq \mathcal S$ be the subgraph consisting of $e_1, \dots, e_\ell$ and  $\mathcal S_v =  \mathcal S \setminus \mathcal S_u$ be the subgraph consisting of $e'_1, \dots, e'_\ell$.
We say that $\mathcal S_u, \mathcal S_v$ are the \emph{halves} of the $(u, v, \ell)$-diamond-chain. 

\begin{lemma}[Diamond-chains from inseparability] \label{lemma:diamond-chains}
    Let $1/k, \eps \gg \gamma \geq \mu \gg 1/\ell_0 \gg \rho \gg 1/n$.
    Let $H$ be an $n$-vertex $k$-graph and let $G$ be its $\gamma$-diamond graph.
    Suppose $G' \subseteq G$ is a $\mu$-inseparable subgraph and let $u, v \in V(G')$ be distinct.
    Then there exists $1 \leq \ell \leq \ell_0$ such that there are at least $\rho n^{k\ell - 1}$ many $(u, v, \ell)$-diamond-chains in $H$.
\end{lemma}

We use the following lemma by Ebsen et al.~\cite[Lemma 2.4]{Ebsen-Maesaka-Reiher-Schacht-Schulke-2020}.

\begin{lemma} \label{lemma:manypaths}
    Let $\mu \gg 1/\ell_0, c, 1/n$.
    Let $G$ be a $\mu$-inseparable graph on $n$ vertices.
    Then for distinct vertices $x, y \in V(G)$,  there is some integer $\ell$ with $0 \leq \ell \leq \ell_0$ such that the number of $(x,y)$-paths with $\ell$ inner vertices is at least $c n^{\ell}$.
\end{lemma}

\begin{proof}[Proof of \cref{lemma:diamond-chains}]
   For any distinct vertices $u, v \in V(G')$, applying \cref{lemma:manypaths} to $G'$, we get that there is some integer $\ell$ with $1 \leq \ell \leq \ell_0$ and  $\mu\gg c\gg \rho$ such that the number of $(u,v)$-paths with $\ell-1$ inner vertices is at least $c n^{\ell-1}$.
   Fix a such $(u,v)$-path with $\ell-1$ inner vertices.
   Recall that $xy \in E(G)$ if and only if there are at least $\gamma \binom{n}{k-1}$ many $(x, y)$-diamonds in $H$; so we can replace each of the $\ell$ edges of the path with a diamond.
   In each step (after discounting for avoiding reusing vertices) we have at least $(\gamma / 2) \binom{n}{k-1}$ many choices, so each $(u,v)$-path with $\ell-1$ inner vertices generates at least $(\gamma/2)^{\ell}\binom{n}{k-1}^{\ell}$ many $(u,v,\ell)$-diamond-chains.
   Summing over all $(u,v)$-paths, we get at least $c n^{\ell-1}(\gamma/2)^{\ell}\binom{n}{k-1}^{\ell}\ge \rho n^{k\ell - 1}$ many $(u, v, \ell)$-diamond-chains in~$H$.
\end{proof}

\section{Complex gadgets} \label{section:complexgadgets}
Assuming we are in the case where the $\gamma$-diamond graph has a $\mu$-separation $\{A, B\}$, we will build absorbers.
As it turns out, in this case, absorbers cannot be built for an arbitrary $k$-set $(w, w_1, \dotsc, w_{k-1})$.
Instead, each vertex $w \in V(H)$ will have a \emph{type} $\pi(w) \in \{0,1\}$, essentially found by looking at the $(k-1)$-sets in $N_H(w)$ and looking if $N_H(w)$ has more $(k-1)$-sets $S$ with $|S \cap A|$ even (in which case $\pi(w) = 0$), or odd (in which case $\pi(w) = 1$).
Given this, we will be able to find an absorbing tuple for $(w, w_1, \dotsc, w_{k-1})$ if and only if $|\{w_1, \dotsc, w_{k-1}\} \cap A| = \pi(w) \bmod 2$.
The absorbing structures will be engineered using diamond-chains and balancers.

Given a $k$-graph $H$ and disjoint sets of the same size $X_1, X_2 \subseteq V(H)$, we say that a subgraph $L \subseteq H$ is an \emph{$(X_1, X_2, t)$-gadget} if
\begin{enumerate}
    \item $X_1 \cup X_2 \subseteq V(L)$;
    \item $L$ is the edge-disjoint union of diamond-chains and $t$-balancers;
    \item $L$ can be decomposed into exactly two edge-disjoint isomorphic subgraphs ${L}_1$ and ${L}_2$;
    \item for each $i \in [2]$, ${L}_i$ consists of vertex-disjoint edges and $t$-stars;
    \item for each $t$-balancer or diamond-chain in $L$, one of its halves is in ${L}_1$ and the other one is in ${L}_2$; 
    \item $X_1=V(L)\setminus V({L}_2)$ and $X_2=V(L)\setminus V({L}_1)$.
\end{enumerate}
Note that $X_i \subset V({L}_i)$ for each $i\in [2]$. We say that $L$ is just a \emph{$t$-gadget} if we do not want to refer to the sets $X_1, X_2$. 

As an example, note that any $(x, y, \ell)$-diamond-chain is an $(\{x\}, \{y\}, t)$-gadget for any $t \geq 1$ (since it contains no balancers of any kind).
Also, balancers naturally give rise to gadgets, as follows.
Suppose that $t_A + t_B = t$ and that $J$ is a $(t_A, t_B)$-balancer, with halves $J_1$ and $J_2$.
Then it is an $(X_1, X_2, t)$-gadget for the sets $X_1 = V(J) \setminus V(J_2)$ and $X_2 = V(J) \setminus V(J_1)$.

We can also combine edge-disjoint gadgets to create new gadgets for different sets.
Note that the first observation increases the sizes of the sets, while the second observation does the opposite.

\begin{observation} \label{observation:gadget-up}
    If $L$ is an $(X_1, X_2, t)$-gadget, $L'$ is an $(X'_1, X'_2, t)$-gadget, and $L, L'$ are vertex-disjoint, then $L \cup L'$ is an $(X_1 \cup X'_1, X_2 \cup X'_2, t)$-gadget.
\end{observation}

\begin{observation} \label{observation:gadget-down}
    If $L$ is an $(X_1, X_2, t)$-gadget, $L'$ is an $(X'_1, X'_2, t)$-gadget, $X'_1 \subseteq X_1$, $X'_2 \subseteq X_2$ and $L, L'$ are vertex-disjoint except for $X'_1 \cup X'_2$, then $L \cup L'$ is an $(X_1 \setminus X'_1, X_2 \setminus X'_2, t)$-gadget.
 \end{observation}

\cref{observation:gadget-up} follows directly from the definition of $t$-gadgets since $L$ and $L'$ are vertex-disjoint. \cref{observation:gadget-down} is a bit more complex.
To see it, suppose that $L$ is an $(X_1, X_2, t)$-gadget that 
can be decomposed into two edge-disjoint isomorphic subgraphs $L_1$ and $L_2$, and 
$L'$ is an $(X'_1, X'_2, t)$-gadget that 
can be decomposed into two edge-disjoint isomorphic subgraphs $L'_1$ and $L'_2$. 
Then $L_1\cup L'_2$ and $L_2\cup L'_1$ are two edge-disjoint isomorphic subgraphs of $L \cup L'$. Moreover, $L_1\cup L'_2$ consists of vertex-disjoint edges and $t$-stars, since $L, L'$ are vertex-disjoint except for $X'_1 \cup X'_2$, 
$X'_1\subseteq X_1\subset V(L_1)$, $X'_1\cap V(L'_2)=\emptyset$, and 
$X'_2\subset V(L'_2)$, $X'_2\subseteq X_2$ and $X_2\cap V(L_1)=\emptyset$. Similarly, $L_2\cup L'_1$ consists of vertex-disjoint edges and $t$-stars.
Finally, since $X_i' \subseteq X_i \subseteq V(L_i)$ and $X_i' \subseteq V(L_i')$ for $i \in [2]$, we have \[X_1 \setminus X_1' = V(L \cup L') \setminus V(L_{2} \cup L_{1}') \quad \text{ and} \quad X_2 \setminus X_2' = V(L \cup L') \setminus V(L_{1} \cup L_{2}').\] 
Thus,  $L \cup L'$ is an $(X_1 \setminus X'_1, X_2 \setminus X'_2, t)$-gadget.

We leverage the above observations to find many $(Y_1, Y_2, t)$-gadgets for some $2$-sets $Y_1, Y_2$ contained in $A$ and $B$, respectively.

\begin{lemma}[Even balancing across sets] \label{lemma:fewdiamonds-pairbalancer}
    Let $k \geq 3$, $t \geq 1$ be integers, and $1/k, \eps \gg \gamma \geq \mu \gg 1/M, 1/n$ and $\eps, 1/t \gg 1/M \gg \rho$.
    Let $H$ be an $n$-vertex $k$-graph with $\delta_{k-1}(H) \geq (1/3 + \eps)n$, and let $G$ be its $\gamma$-diamond graph.
    Suppose $G$ has a $\mu$-separation $\{A, B\}$ such that $G[A], G[B]$ are $\mu$-inseparable.
    Then there exists $m \leq M$ such that there are at least $\rho n^{m}$ many $m$-vertex $t$-gadgets in $H$; and for each of those gadgets $L$, there exist $Y_1 \subseteq A$, $Y_2 \subseteq B$ of size $2$ such that $L$ is a $(Y_1, Y_2, t)$-gadget.
    
    Moreover, if in addition $t$ is even, then we get the same outcome for sets $Y_1, Y_2$ of size one each; and moreover, $L$ contains exactly one $t$-balancer.
\end{lemma}

\begin{proof}
    Let $\eps, 1/t, \mu \gg 1/L \gg 1/M \gg \rho_1, \rho_2 \gg \rho$.
    Let $t_A = \lfloor t/2 \rfloor$, $t_B = \lceil t/2 \rceil$, and $m_1 = (2k-3)t+2$.
    By \Cref{lemma:balancers}, there are at least $\rho_1 n^{m_1}$ many $m_1$-vertex $(t_A, t_B)$-balancers in $H$.
    Let $J$ be one of those balancers, with halves $J_1$ and $J_2$.
    Then, as discussed, $J$ is an $(X_1, X_2, t)$-gadget for the sets $X_1 = V(J_1) \setminus V(J_2)$ and $X_2 = V(J_2) \setminus V(J_1)$.
    Recalling the construction of $(t_A, t_B)$-balancers, we have 
\[
|X_2\cap B|-|X_1\cap B|=1 + t_B - t_A=\begin{cases} 
2 & \text{for odd }  t , \\
1 & \text{for even } t.
\end{cases}
\]
Let $s=|X_1|=|X_2|$, $r = |X_1 \cap B|$ and $r + c = |X_2 \cap B|$, where $c = 2$ or $c = 1$ if $t$ is odd or even, respectively.
Suppose $X_i = \{ x^i_1, \dotsc, x^i_s\}$, such that $X_1 \cap B = \{ x^1_1, \dotsc, x^1_{r} \}$ and $X_2 \cap B = \{ x^2_1, \dotsc, x^2_{r+c} \}$.
    Thus for $j \in [r]$, we have that $x^1_j, x^2_j$ belong to $B$, and for $j \in [s] \setminus [r+c]$, we have that $x^1_j, x^2_j$ belong to $A$.
    
    Since both $x^1_1$ and $x^2_1$  belong to $B$ and $G[B]$ is $\mu$-inseparable by assumption, by \Cref{lemma:diamond-chains} there exists $\ell \leq L$ such that there are at least $2 \rho_2 n^{k\ell -1}$ many $(x^1_1, x^2_1, \ell)$-diamond-chains in $H$.
    Of these, at least $\rho_2 n^{k\ell -1}$ are vertex-disjoint with $J$ (apart from $x^1_1, x^2_1$).
    By \Cref{observation:gadget-down}, the union of $J$ with any of those diamond-chains gives an $(X_1 \setminus \{x^1_1\}, X_2 \setminus \{x^2_1\}, t)$-gadget.
    Iterate this, finding an $(x^1_j, x^2_j, \ell_j)$-diamond-chain for every $j \in \{2, \dotsc, r\}$ and for every $j \in \{r+c+1, \dotsc, s\}$.
    Taking the union, we obtain a $(Y_1, Y_2, t)$-gadget, where $Y_i = \{ x^i_{r+1}, \dotsc, x^i_{r+c} \}$, and we have $Y_1 \subseteq A$ and $Y_2 \subseteq B$.

    If $t$ is odd, then $c = 2$; thus the above process creates the required gadgets.
    If $t$ is even, then $c = 1$; thus the above process creates gadgets between sets $Y_1, Y_2$ of size $1$ each (note that this also gives immediately the `moreover' part).
    Then taking two vertex-disjoint gadgets as above and using \Cref{observation:gadget-up} gives gadgets for suitable sets of size $2$, as desired.
\end{proof}

\subsection{Parity-gadgets}

We will build many gadgets between two arbitrary $(k-1)$-sets as long as their intersection with $A$ has the same parity.

\begin{lemma}[Gadgets from few diamonds]\label{Gadgets from few diamonds}
    Let $k \geq 3$,  $1/k, \eps \gg \gamma \geq \mu \gg 1/M \gg 1/n$ and $\eps,  1/t \gg \rho$.
    Let $H$ be an $n$-vertex $k$-graph with $\delta_{k-1}(H) \geq (1/3 + \eps)n$, and let $G$ be its $\gamma$-diamond graph.
    Suppose $G$ has a $\mu$-separation $\{A, B\}$, such that $G[A], G[B]$ are $\mu$-inseparable.
    Let $X_1, X_2$ be $(k-1)$-sets such that $|X_1 \cap A| \equiv |X_2 \cap A| \bmod 2$.
    Then there exists some $m \leq M$ such that there are $\rho n^{m - 2(k-1)}$ many $m$-vertex $(X_1, X_2, t)$-gadgets in $H$.
\end{lemma}

\begin{proof}
    Without loss of generality, we can assume $|X_1 \cap A| \leq |X_2 \cap A|$.
    Since $|X_1 \cap A| \equiv |X_2 \cap A| \bmod 2$, we have that $|X_2 \cap A| - |X_1 \cap A|$ must be even.
    We will prove the result by induction on $|X_2 \cap A| - |X_1 \cap A|$.

    We set up notation to formally carry the induction procedure.
    Let $\ell_0, c$ be such that $\mu \gg 1/\ell_0, c$, so \Cref{lemma:diamond-chains} is valid with the choice of $\mu, \ell_0, c$.
    For each $0 \leq s \leq k-1$, let \[ m_s \gg \dotsb \gg m_0 \gg 1/\mu , \qquad \rho_s \ll \dotsb \ll \rho_0 \ll \eps, 1/t, \mu. \]
    We will use induction to show, for each $s \geq 0$, that any two disjoint $(k-1)$-sets $X_1, X_2 \subseteq V(H)$ with $|X_1 \cap A| \equiv |X_2 \cap A| \bmod 2$ and $|X_2 \cap A| - |X_1 \cap A| = 2 s$, there exists $m \leq m_s$ such that there are at least $\rho_s n^{m-2(k-1)}$ many $m$-vertex $(X_1, X_2, t)$-gadgets. 
    This allows us to conclude with $M := \max\{ m_s : 0 \leq 2s \leq k-1 \}$ and $\rho := \min \{ \rho_s : 0 \leq 2s \leq k-1 \}$.

    We prove the base case $s=0$.
    Let $X_1, X_2 \subseteq V(H)$ be two disjoint $(k-1)$-sets such that $|X_1 \cap A| = |X_2 \cap A| = r$.
    In this case, we can construct $(X_1, X_2, t)$-gadgets using only diamond-chains.
    We will describe the construction of one such structure by adding substructures step by step, and then argue that we have many choices in each step, to this process yields many $(X_1, X_2, t)$-gadgets.
    For each $i \in [2]$, let $X_i = \{ x^i_1, \dotsc, x^i_{k-1}\}$ where $X_i \cap A = \{ x^i_1, \dotsc, x^i_r\}$.
    For each $j \in [k-1]$, since $x^1_j$ and $x^2_j$ belong to the same part, we add an $(x^1_j, x^2_j, \ell)$-diamond-chain that is vertex-disjoint from $X_1 \cup X_2$ and from all the diamond-chains found so far.
    By \Cref{observation:gadget-up}, this process creates an $(X_1, X_2, t)$-gadget.
    \Cref{lemma:diamond-chains} ensures that each diamond-chain has at most $k(\ell_0+1)-1$ extra vertices, so in total our structure has at most $|X_1|+|X_2|+(k-1)(k(\ell_0+1)-1) \leq 2k+k^2(\ell_0+1) \leq m_0$ vertices.
    Suppose now we need to choose the $j$-th diamond-chain.
    \Cref{lemma:diamond-chains} ensures that there is some $\ell_j \leq \ell_0$ and at least $c n^{ k (\ell_j+1)-1}$ many $(x^1_j, x^2_j, \ell_j)$-diamond-chains.
    Since we need to avoid at most $2k+k^2(\ell_0+1)$ vertices, there are still, say, at least $c n^{ k (\ell_j+1)-1} / 2$ possible choices for the $j$-th step.
    This process generates $(X_1, X_2, t)$-gadgets with exactly $m'_0:= 2(k-1) + \sum_{j=1}^{k-1} (k(\ell_j+1)-1) ) \leq m_0$ many vertices, and by construction, the number of choices is at least
    \[ \prod_{j=1}^{k-1} \left( \frac{c}{2} n^{ k (\ell_j+1) - 1}\right) = \left( \frac{c}{2} \right)^{k-1} n^{m'_0 - 2(k-1)} \geq \rho_0 n^{m'_0 - 2(k-1)}, \]
    as desired. This proves the base case of the induction.
    
    Now suppose that $|X_2 \cap A| - |X_1 \cap A| = 2s > 0$.
    As before, we describe the construction of the $(X_1, X_2, t)$-gadgets and then argue about the number of choices.
    By \Cref{lemma:fewdiamonds-pairbalancer}, there exists $ M'<M$ such that there are at least $\rho_0 n^{M'}$ many $M'$-vertex $t$-gadgets in $H$.
    Fix one of those $t$-gadgets $L$, which is a $(Y_1, Y_2, t)$-gadget for some $2$-sets $Y_1, Y_2$ contained in $A$ and $B$, respectively.
    We first claim that $|A|,|B|\geq n/20$. Towards a contradiction,  we suppose $|A|<n/20$ and $|B|> 19n/20$. For any $x\in A$, by \cref{lemma:diamondgraph-mindeg}, $\deg_{G}({x},B)\ge \delta(G)-|A|\ge n/20$.
    So $e_G(A,B)\ge n|A|/20>\mu |A||B|$, which is a contradiction.
    Select an arbitrary set $Z$ of size $k-3$, such that $|Z \cap A| = |X_2 \cap A|-2$, and disjoint from $X_1, X_2$ and $V(L)$.
    The number of choices for such a set $Z$ is at least $\binom{|A|-M'-2k}{|X_2 \cap A|-2}\binom{|B|-M'-2k}{k-1-|X_2 \cap A|}\ge \eps n^{k-3}$.
    Let $Z_i = Z \cup Y_i$, for $i \in [2]$.
    Note that $Z_2$ is a $(k-1)$-set with $|Z_2 \cap A| = |Z \cap A| = |X_2 \cap A|-2$, so we have $|Z_2 \cap A| - |X_1 \cap A| = 2(s-1)$. 
    By induction, we can find $\rho_{s-1} n^{m_{s-1}-2(k-1)}$ many $(X_1, Z_2, t)$-gadgets; let $L_1$ be one that is vertex-disjoint from $Y_1, X_2$ and $V(L) \setminus Z_2$.
    Since $|Z_1 \cap A| = |Z_2 \cap A| + 2 = |X_2 \cap A|$, again by induction we can find $\rho_{s-1} n^{m_{s-1}-2(k-1)}$ many $(Z_1, X_2, t)$-gadgets, and selecting one denoted by $L_2$ that is vertex-disjoint from the structures constructed so far.
    Let $m_s:=M'+2m_{s-1}-k-1$.
    Thus the union $L \cup L_1\cup L_2$ gives an $(X_1, X_2, t)$-gadget with $m_s$ vertices, as desired.
    Summarizing, the number of $(X_1, X_2, t)$-gadgets is at least  \[ \rho_0 n^{M'}\eps n^{k-3}(\rho_{s-1} n^{m_{s-1}-2(k-1)})^2\ge \rho_s n^{M'+2m_{s-1}-k-1-2(k-1)}=\rho_s n^{m_{s}-2(k-1)},\]
    as desired. This finishes the proof.
\end{proof}

\subsection{Parity-absorbers}

Let $H$ be a $k$-graph, $w \in V(H)$ and $S\subseteq V(H)$ be a $(k-1)$-set of $V(H)$.
A \emph{$t$-absorbing gadget} for $(w, S)$ consists of a $(k-1)$-set $W \in N_H(w)$, together with a $(W, S, t)$-gadget $L$ that is disjoint from $w$. Suppose $L$ is decomposed into exactly two edge-disjoint isomorphic subgraphs ${L}_1$ and ${L}_2$ such that $S\subset V(L_1)$ and $W\subset V(L_2)$.   
We call $L_2$ a \emph{parity-absorber}  for $(w, S)$. 
Let $m := |V(L)| - k + 1$ be the \emph{size} of the parity-absorber.
For $w \in V(H)$, a $(k-1)$-set $S \subseteq V(H)$, and $m \geq 0$, let $\mathcal{A}^{(t)}_m(w, S)$ be the set of all parity-absorbers of size $m$  within $t$-absorbing gadgets in $H$.

\begin{lemma}[Parity-absorbers from few diamonds] \label{Parity-absorbers}
    Let $k \geq 3$,  $1/k, \eps \gg \gamma \geq \mu \gg 1/M \gg 1/n$ and $\eps,  1/t \gg \rho$.
    Let $H$ be an $n$-vertex $k$-graph with $\delta_{k-1}(H) \geq (1/3 + \eps)n$, and let $G$ be its $\gamma$-diamond graph.
    Suppose $G$ has a $\mu$-separation $\{A, B\}$, such that $G[A], G[B]$ are $\mu$-inseparable.
    Then there exists $\pi: V(H) \rightarrow \{0,1\}$ such that for every $w \in V(H)$, and every $(k-1)$-set $S \subseteq V(H)$ with $|S \cap A| \equiv \pi(w) \bmod 2$, there exists $m \leq M$ such that $|\mathcal{A}^{(t)}_m(w, S)| \geq \rho n^m$.
\end{lemma}

\begin{proof}
    We begin by defining the function $\pi$.
    Let $w \in V(H)$ be arbitrary.
    Since $\delta_{k-1}(H) \geq (1/3 + \eps) n$, we have $|N_H(w)| \geq (1/3 + \eps) \binom{n-1}{k-1}$.
    If there are at least $\binom{n-1}{k-1}/6$ many $(k-1)$-sets $W \in N_H(w)$ such that $|W \cap A| \equiv 0 \bmod 2$, then set $\pi(w) = 0$, otherwise set $\pi(w) = 1$.

    Next, given a vertex $w$, if a $(k-1)$-set $S \subseteq V(H)$ satisfies $|S \cap A| = \pi(w) \bmod 2$, by our choice of $\pi(w)$ there are at least $\binom{n-1}{k-1}/6$ many $(k-1)$-sets $W \in N_H(w)$ such that $|S \cap A| \equiv |W \cap A| \bmod 2$.
    For each such choice of $W$, by \cref{Gadgets from few diamonds}, there are at least $\rho' n^{|V(L)|- 2(k-1)}$ many  $(W,S,t)$-gadgets $L$ with $|V(L)|\le M$ and $\eps,1/t\gg\rho'\gg \rho $.
    Let $m := |V(L)| - k + 1$.
    Summing over all $\binom{n-1}{k-1}/6$ many $(k-1)$-sets $W \in N_H(w)$, we get that $|\mathcal{A}^{(t)}_m(v, S) |\geq\rho' n^{|V(L)|- 2(k-1)} \binom{n-1}{k-1}/6\ge\rho n^m$.
\end{proof}

\section{Immersion Lemma} \label{section:immersion}
In this section, we show how to cover our gadgets (absorbers, balancers, etc.) with an embedding of a small tree.
Each of our gadgets consists of edge-disjoint unions of diamonds and $k$-expansion of $K_{2,t}$ for some integer $t$.
As explained in \Cref{ssection:sketch-gadgets}, to be able to perform switches we need to have an embedding that covers `half' of the edges in our gadget: either half of the edges of each diamond or half of a $k$-expansion of $K_{2,t}$.

We introduce language to formalize this discussion.
Note that a $t$-star with a center vertex $r$, denoted by $S_r$, is a $k$-graph with $t$ edges that are disjoint except for $r$.
A \emph{half-diamond} is simply a $1$-star together with a center vertex $r$, called the \emph{center} of the half-diamond.
Now suppose $T$ is a tree, and $\phi$ is an embedding of $T^{(k)}$ in a $k$-graph $H$.

\begin{enumerate}
    \item Given a $t$-star $S_r \subseteq H$, we say that $S_r$ is \emph{immersed in $\phi(T^{(k)})$} if there exists a vertex $v \in V(T)$ of degree $t$ such that $\phi(v) = r$, and all the $t$ edges in $T^{(k)}$ adjacent to $v$ are mapped to distinct edges of $S_v$;

    \item Given a half-diamond $S$ with center $r$, we say that $S$ is \emph{immersed in $\phi(T^{(k)})$} if there exists an edge $uv \in E(T)$ such that its $k$-expansion $e \in E(T^{(k)})$ is mapped to $S$ via $\phi$, and $r \notin \{ \phi(u), \phi(v)\}$.
\end{enumerate}

The next lemma shows that any sufficiently small collection of $t$-stars and half-diamonds can be immersed by an embedding of a small $k$-expansion hypertree.

\begin{lemma} \label{lemma:immersion}
Let $k \ge 2$ with $1/k, \Delta, \eps \gg \alpha \gg \beta \gg 1/n$.
Let $H$ be a $k$-graph on $n$ vertices
satisfying
\[
\delta_d(H)\ge \begin{cases} 
(\frac{1}{2^{k-1}}+\eps) \binom{n-d}{k-d} & \text{for } d=1, \\
\eps \binom{n-d}{k-d} & \text{for }  2\le d\le k-1.
\end{cases}
\]
Let $T$ be a tree on $\alpha n$ vertices and with $\Delta (T)\leq \Delta$, and let $t \in [\Delta]$ be such that there are at least $|V(T)|/\Delta$ vertices of degree $t$ in $T$.
Let $\mathcal{A}$ be a family of at most $\beta n$ vertex-disjoint subgraphs of $H$ each of which is a $t$-star or a half-diamond.
Then there is an embedding $\phi$ of $T^{(k)}$ in $H$ such that every subgraph of $\mathcal{A}$ is immersed in $\phi(T^{(k)})$.
\end{lemma}

\begin{proof}
Set $m:=|\mathcal A|$, by assumption we have $m \le \beta n$.
Suppose that $\mathcal{A}$ contains $m_1$ many $t$-stars $S'_{y_i}$ with center vertex $y_i$ for $i\in [m_1]$; and also contains $m-m_1$ half-diamonds $S'_{y_i}$  with center vertex $y_i$ for $i\in [m]\setminus [m_1]$.
We will consecutively define subtrees $T_1 \subseteq \dotsb \subseteq T_m$ of $T$ and embeddings $\phi_1, \dots, \phi_m$ such that $\phi_i: T_i^{(k)}\to H$ is an embedding of $T_i^{(k)}$ in $H$ and $S'_{y_1}, \ldots, S'_{y_{i}}$ are immersed in $\phi_i(T^{(k)}_i)$ for each $i\in [m]$.

Let $U\subseteq V(T)$ be the vertices of degree $t$ in $T$.
By assumption, $|U|\ge \alpha n/\Delta$.
Let $T'$ be the $4$-power of $T$; which is the graph obtained from $T$ by adding an edge between every pair of vertices of distance at most $4$ in $T$.
Then $\Delta (T')\leq \sum^{4}_{i=1} \Delta^{i}$.
A greedy argument shows that the induced subgraph $T'[U]$ has an independent set $I$ of size at least $\frac{|U|}{\Delta (T')+1}\ge \alpha' n$, where $\alpha'$ is a constant with $\alpha\gg \alpha'\gg \beta$.
Note that, by construction, the vertices in $I$ have all degree $t$ and are pairwise at distance at least $5$ in $T$.

Let $x_0 \in I$ be arbitrary and consider the rooted tree $(T, x_0)$. 
We label the vertices in $I$ by $x_0, x_1, \dots, x_{|I|-1}$, using the order given by a breadth-first search in $T$ started in $x_0$.
In particular, the vertices $x_1, \dotsc, x_{|I|-1}$ are ordered by non-decreasing depth, i.e., distance from $x_0$.
Since $|I| \gg \beta n \geq m$ and $I \subseteq U$, we can find $m$ vertex-disjoint $t$-stars $S''_{x_i} \subseteq T$ with center vertex $x_i$, one for each $i\in [m]$.
Since vertices in $I$ have pairwise distance at least five, any two vertices from different stars $S''_{x_i}, S''_{x_j}$ must have distance at least three.
For each $i \in [m_1]$, let $S_{x_i} = S''_{x_i}$; and for $i \in [m] \setminus [m_1]$, let $S_{x_i}$ be the edge in $S''_{x_i}$ that joins $x_i$ with its parent in $(T, x_0)$.
By definition, $S^{(k)}_{x_1}, \dots, S^{(k)}_{x_m}$ are $m$ vertex-disjoint subgraphs of the $k$-expansion $T^{(k)}$ of $T$,  where each $S^{(k)}_{x_i}$ is a $t$-star for $i \in [m_1]$, or a single edge for $i \in [m] \setminus [m_1]$. 
For $j\in [m]\setminus [m_1]$, let $f_j$ be the expansion edge of $S^{(k)}_{x_j}$; and let $u_j\in f_j$ be an expansion vertex.
Clearly, $u_j\neq x_j$.

Define $T_0 := \{ x_0 \}$, and note that $T_0 \subseteq T$.
Given $i\in [m]$, we let $T_{i}$ be the union of $T_{i-1}$, the shortest path $P_{i}$ from $T_{i-1}$ to $x_i$, and $S_{x_{i}}$.
By construction, $S_{x_i}\subseteq T_{i}$ and $T_{i-1} \subset T_{i}$.
By our ordering of $x_1, \dotsc, x_{m}$ and the fact that vertices in $I$ have pairwise distance at least five; the length of each path $P_{i}$ is at least three.

Now we define our embeddings $\phi_1, \dotsc, \phi_m$.
Define $V(\mathcal A)=\bigcup^{m}_{i=1}V(S'_{y_i}) \subseteq V(H)$ as the set of vertices participating in structured that need to be immersed by our embedding.
Initially, let $y_0 \notin V(\mathcal{A})$ be arbitrary, and let $\phi_0(x_0) = y_0$,
this is an embedding of $T^{(k)}_0$.
Next, suppose $i \in [m]$ is given and that the embedding $\phi_{i-1}$ of $T^{(k)}_{i-1}$ has been defined.
We now define $\phi_{i}$ by extending $\phi_{i-1}$ as follows.

We first define the image of $S^{(k)}_{x_i}$.
If $i\in [m_1]$, set $\phi_{i}(x_{i})=y_{i}$, and for each $e \in E(S^{(k)}_{x_{i}})$, set $\phi_{i}(e)$ to be a distinct edge of $S'_{y_{i}}$.
On the other hand, if $i\in [m]\setminus [m_1]$ we set $\phi_{i}(u_{i})=y_{i}$ and $\phi_{i}(f_{i})=S'_{y_{i}}$.

To complete the definition of $\phi_i$, we need to specify the embedding of the vertices in $P^{(k)}_i$.
Recall that $P_i$ is a path from $T_{i-1}$ to $x_i$ with length at least three; so $P^{(k)}_{i}$ is a loose path from $T^{(k)}_{i-1}$ to $x_{i}$ with at least three edges.
Observe that by our definition of $S_{x_i}$, in any case we have that the last edge of $P^{(k)}_i$ belongs to $S^{(k)}_{x_i}$, so we have already defined its embedding.
Let $Q_i$ be the subpath of $P_i$ after removing its final edge that contains $x_i$; so $Q_i$ is a path of length $\ell \geq 2$.
Observe also that $Q_i$ has length at most $|V(T)| \leq \beta n$.
Then $Q^{(k)}_i$ is a loose path of length $\ell$ satisfying $2 \leq \ell \leq \beta n$.
Suppose $Q^{(k)}_{i} = v^{i}_1 v^{i}_2 \dotsb v^{i}_{(k-1)\ell+1}$, where $v^{i}_1\in V(T^{(k)}_{i-1})$ and $v^i_{(k-1)\ell+1} \in V(S^{(k)}_{x_i})$ is an anchor vertex of $T^{(k)}$.
Next, in $H$ we will find a loose path of length $\ell$ between $x := \phi_{i}(v^{i}_1)$ and $y := \phi_{i}(v^{i}_{(k-1)\ell+1})$ such that all interior vertices belong to $V(H) \setminus (V(\phi_{i}(T^{(k)}_i)) \cup V(\mathcal A))$.
Having such a path, we can map $Q^{(k)}_i$ to it, which will complete the definition of $\phi_{i}$. 

Let $H':= H-(V(\phi_{i}(T^{(k)}_i))\cup V(\mathcal A)) \setminus \{x, y\}$. 
When passing from $H$ to $H'$, we remove at most $|V(T^{(k)})|+|V(\mathcal{A})| \leq \alpha k n + \beta \Delta k n \leq \eps n/2k$ vertices, so we have that $H'$ satisfies
\[
\delta_d(H')\ge \delta_d(H) - \frac{\eps n}{2k}\binom{n-d-1}{k-d-1}\ge 
\begin{cases} 
(\frac{1}{2^{k-1}} + \frac{\eps}{2} )  \binom{n-d}{k-d} & \text{for } d=1, \\
\frac{\eps}{2} \binom{n-d}{k-d} & \text{for }  2\le d\le k-1.
\end{cases}
\]
When $2\le d\le k-1$, using the minimum $2$-degree condition of $H'$ and $\ell \leq \beta n$, we can always greedily find a loose path of length $\ell$ between $x$ and $y$ while avoiding the already used vertices.
For the $d=1$ case, choose a constant $\delta$ such that $\eps \gg \delta\gg \beta, \alpha$.
Recall that $\ell \geq 2$.
Using the $1$-degree condition $\delta_1(H') \geq \frac{\eps}{2} \binom{n-1}{d-1}$, we can greedily find a loose path of length $\ell-2$ starting at $x$, and ending at some vertex $x' \in V(H')$.
Such a path uses at most $\ell k \leq \beta k n$ vertices.
It only remains to find a loose $(x', y)$-path in $H'$ of length two, not using any of the vertices used so far in our path.
By \cref{lemma:shortpath}, there are at least $\delta |V(H')|^{2k-3}$ $(x', y)$-paths in $H'$ of length two.
Since $\delta\gg \beta, \alpha$, there is indeed an $(x', y)$-path avoiding the already used vertices, which gives the desired path.

Having finished the definition of $\phi_{i}$, note that $\phi_{i}$ is an embedding of $T^{(k)}_{i}$ into $H$, and $S'_{y_1},\dots, S'_{y_{i}}$ are immersed in $\phi_i(T^{(k)}_{i})$.
At the end of this process, we have an embedding $\phi_{m}: T^{(k)}_m \to H$ that immerses all the subgraph of $\mathcal A$.
We then extend $\phi_{m}$ to an embedding $\phi$ of $T^{(k)}$ in $H$ by embedding the extra edges as leaves, one by one. This is possible because at every step, the minimum degree of any vertex of $H$ into the set of unoccupied vertices is at least $\eps\binom{n-1}{k-1}/2$, so an embedding of the next edge can be chosen greedily.
\end{proof}

\section{Embedding spanning expansion trees with one even-degree vertex} \label{section:proof-even}

\begin{proof}[Proof of \cref{theorem:main-even}]
We choose constants $\gamma, \mu, \nu, \alpha, \rho_1, \rho_2, \beta, \delta, M$ so that
\begin{equation}\label{equa-3}
1/k, \eps, 1/\Delta \gg \gamma, \nu \gg \mu \gg 1/M, \beta \gg \rho_1 \gg \alpha  \gg \delta \gg  \rho_2\gg  \delta' \gg 1/n > 0. 
\end{equation}
Let $T$ be an $n$-vertex tree with $\Delta_1(T) \leq \Delta$ and at least one even-degree vertex.
Let $T^{(k)}$ be the $k$-expansion of $T$, and let $N := |V(T^{(k)})| =(k-1)n-k+2$.
Let $H$ be a $k$-graph on $N$ vertices with $\delta_{k-1}(H) \geq (1/3 + \eps)N$, and let $G$ be the $\gamma$-diamond graph of $H$. By \cref{lemma:diamondgraph-mindeg}, \cref{lemma:diamondgraph-indep} and \cref{lemma:diamonggraph-separation}, either $G$ is $\mu$-inseparable, or there is a $\mu$-separation $\{A, B\}$ of $V(G)$ such that both $G[A]$ and $G[B]$ are $\mu$-inseparable.

\subsection{The inseparable case}
We first address the case where the $\gamma$-diamond graph $G$ of $H$ is $\mu$-inseparable. The situation is relatively simple in this case, and we only provide a proof sketch since all necessary elements are similar to those in the proof of \Cref{theorem:main}.

Since $G$ is $\mu$-inseparable, \Cref{lemma:diamond-chains} implies that every vertex pair $u, v \in V(H)$ admits at least $\rho' N^{k \ell - 1}$ many $(u, v, \ell)$-diamond-chains in $H$, where $\ell \leq M$ is an integer and $\rho'$ is a constant satisfying $\mu \gg 1/\ell \gg \rho' \gg \rho_1$. 

Fix a $k$-set of distinct vertices $(w, w_1, \ldots, w_{k-1})$ of $H$. By \Cref{degree}, there are at least $(1/3 + 3\eps/4)\binom{N-1}{k-1}$ many edges $w v_1 \ldots v_{k-1}$ disjoint from $\{w_1, \ldots, w_{k-1}\}$. For each $i \in [k-1]$, we have that there is $\ell_i \leq M$ and at least $\rho' N^{k \ell_i - 1}$ many $(w_i, v_i, \ell_i)$-diamond-chains in $H$. 
Recall that $\mathcal{S}_{v_i}$ is the half of the $(w_i, v_i, \ell_i)$-diamond-chain. We continue to use the terminology from \Cref{section:proof-main} and say that an absorbing tuple for $(w, w_1, \ldots, w_{k-1})$ is a tuple $(\mathcal{S}_{v_1}, \ldots, \mathcal{S}_{v_{k-1}})$ satisfying those halves are vertex-disjoint and $w  v_1 \ldots v_{k-1} \in E(H)$.
Let $A_m(w, w_1, \ldots, w_{k-1})$ denote the set of all $m$-vertex absorbing tuples for $(w, w_1, \ldots, w_{k-1})$; and let $A_m(H)$ denote the union of $A_m(w, w_1, \ldots, w_{k-1})$ for all $k$-sets $(w, w_1, \ldots, w_{k-1})$ of distinct vertices of $V(H)$. 
By our discussion above, for each $k$-tuple $(w, w_1, \dotsc, w_{k-1})$ there is some $1 \leq m \leq k^2M-1$ such that there are at least $(\rho'/2)n^m$ many absorbing tuples in $A_m(w, w_1, \dotsc, w_{k-1})$.
Let $A(H) = \bigcup_{1 \leq m \leq k^2M-1} A_m(H)$.

Using a similar probabilistic argument as in the proof of \Cref{lem:counting-absorb}, we can get a set $\mathcal{A} \subseteq A(H)$ such that $\mathcal{A}$ contains at most $\beta N$ pairwise vertex-disjoint absorbing tuples and $|A(w, w_1, \ldots, w_{k-1}) \cap \mathcal{A}| \geq \rho_1 N$, for each choice of $(w, w_1, \dotsc, w_{k-1})$. Similarly, we can find a small subtree $T(u)$ such that $T^{(k)}(u)$ has size at most $\Delta \nu N$.
Note that $\mathcal{A}$ consists of vertex-disjoint half-diamonds. 
By \Cref{lemma:immersion}, there is an embedding $\phi_0: T^{(k)}(u) \to H$ such that every element in $\mathcal{A}$ is immersed in $\phi_0(T^{(k)}(u))$. Suppose that $\phi_0(u) = v$. Let $T_0^{(k)} := T^{(k)} - V(T^{(k)}(u)) \setminus \{u\}$ and $H_0 := H - V(\phi(T^{(k)}(u))) \setminus \{v\}$. 
Clearly, $\delta_{k-1}(H_0) \geq (1/3 + \eps/2)|V(H_0)|$.
Remove leaves repeatedly from rooted $k$-expansion hypertree $(T_0^{(k)}, u)$ to get a subtree $T_1^{(k)}$ such that $|V(T_0^{(k)})| - |V(T_1^{(k)})| =\delta N$.
By~\cref{almost-cover}, there is an embedding $\phi: T_1^{(k)} \to H_0$ such that $\phi (u) = v$. 
Finally, similar to \Cref{absorb-lem}, we can absorb the remaining vertices (the only difference being that we switch all the diamonds of each $(w_i, v_i, \ell)$-diamond-chain, rather than a single diamond as in \Cref{absorb-lem}).
Therefore, we can get an embedding from $T^{(k)}$ to $H$. 

\subsection{The separable case}
Now assume that the $\gamma$-diamond graph $G$ of $H$ is $\mu$-separable, i.e., there is a partition $\{A, B\}$ of $V(G)$ such that $G[A]$ and $G[B]$ are $\mu$-inseparable.
We will show that $T^{(k)} \subseteq H$.
    \medskip 

    \noindent \emph{Step 1: Partitioning the tree.}
    We begin by carefully dividing the tree $T$ into parts, considering its even-degree vertices.
    Let $u' \in V(T)$ be any even-degree vertex (such a vertex must exist by assumption).
    Consider the rooted tree $(T,u')$.
    Since $\Delta_1(T) \leq \Delta$, there must exist a children $c'$ of $u'$ such that $|V(T(c'))| \geq (|V(T)|-1)/\Delta \geq 2 \Delta \nu n$.
    Now, let $v \in V(T(c'))$ be a maximal-depth vertex such that $\nu n \le |V(T(v))|$.
    The choice of $v$ ensures that \[ \nu n \leq |V(T(v))| \leq \Delta \nu n.\]
    This also implies that $c' \neq v$, since $|V(T(c'))| \geq 2 \Delta \nu n > |V(T(v))|$.

    Let $x$ be the parent of $v$.
    By our discussion above, $x \neq u'$.
    Let $T_0 := T(v)$.
    Let $u$ be the even-degree vertex in the $(u',x)$-path in $T$, excluding $x$, that is closest to $v$ in $T$ (it is of course possible that $u = u'$).
    Now we root $T$ at $u$.
    Let $t^\star$ be the degree of $u$.
    Note that $t^\ast$ is even and at most $\Delta$, by assumption.
    Let $c_1, \dotsc, c_{t^\star}$ be the children of $u$.
    Without loss of generality, we can assume that $v \in V(T(c_1))$.
    This implies that $T_0 \subseteq T(c_1)$.
    Note that our choice of $u$ implies that $v \neq c_1$ and $x \in V(T(c_1))$.
    
    For each $i \in [t^\star] \setminus \{1\}$, let $T_i = T(c_i)$; and let $T_1 = T(c_1) - V(T_0)$.
    Since $v \neq c_1$ and $x \in V(T(c_1)) \setminus V(T_0)$, we have that $V(T_1)$ is non-empty.
    
    Next, we will select a subtree $T_{\mathrm{main}} \subseteq T - V(T_0)$ such that
    \stepcounter{propcounter}
    \begin{enumerate}[label = {{\rm (\Alph{propcounter}\arabic{enumi})}}]
        \item contains $u$ and $x$;
        \item contains $c_1, \dotsc, c_{t^\star}$; and
        \item has $(1 - \delta)n - |V(T_0)|$ vertices.
    \end{enumerate}
    We find such a subtree as follows.
    If there exists some $i \in [t^\star] \setminus \{1\}$ such that $|V(T_i)| > \delta n$, we can obtain $T_{\mathrm{main}}$ from $T - V(T_0)$ by iteratively removing $\delta n$ leaves inside $T_i$.
    Otherwise, we can assume that $|V(T_i)| \leq \delta n$ for all $i \in [t^\star] \setminus \{1\}$.
    Hence, \[|V(T_1)| = n - 1 - |V(T_0)| - \sum_{i=2}^{t^\star} |V(T_i)| \geq n - \Delta \nu n - \Delta \delta n \geq 2n/3.\]    
    Let $P = v_0 v_1 \dotsb v_t$ be the $(x,u)$-path in $T$, with $v_0 = x$ and $v_t = u$ (since $x \neq u$, we have $t \geq 1$).
    If $V(T_1) \setminus V(P)$ has at least $\delta n$ vertices, we can find a $T_{\mathrm{main}}$ again by removing iteratively leaves from $V(T_1) \setminus V(P)$.
    Assuming $|V(T_1) \setminus V(P)| < \delta n$, together with the lower bound on $|V(T_1)|$, this implies that $t+1 = |V(P)| \geq 2n/3 - \delta n - 2 \geq 3 \delta n$.
    Now recall that by our choice of $x$ and $u$, for every $j \in [t-1]$ we have that the degree of $v_j$ in $T$ is at least three.
    Hence, for each $j \in [t-1]$ there exists $u_j \in N_T(v_j) \setminus V(P)$.
    Then we have that
    \[ |V(T_1) \setminus V(P)| \geq t-1 \geq 3 \delta n - 2 > \delta n, \]
    a contradiction; thus we can always find the desired subtree $T_{\mathrm{main}}$.
    \medskip

    \noindent \emph{Step 2: Using the even-degree vertex.}
    For each $i \in \{0\} \cup [t^\star]$, let $T^{(k)}_i$ be the $k$-expansion of $T_i$.
    For each $i \in [t^\star]$, let $e_i \in E(T^{(k)})$ be the edge containing $u$ and $c_i$.
    Note that $x \in V(T_1)$ is the unique neighbor of $v$ in $T_1$. Let $e_0 \in E(T^{(k)})$ be the edge containing $v$ and $x$.
    Note that the $k$-expansion trees $T^{(k)}_i + e_i$ form an edge-decomposition of $T^{(k)}$, where $i\in \{0\}\cup [t^\star]$.

We first choose a reservoir. 
By \cref{res-lem}, there exists a subset $R \subseteq V(H)$ of size $\delta' N$ such that for every pair $x, y \in V(H)$, we have $\deg_H (\{x, y\}, R) \ge \delta' (\delta' N)^{k-2} / 2$.
Let $H_0 := H - R$ and $|V(H_0)|:=N_0$. The choice $\eps\gg \delta'$ ensures that $H_0$ satisfies
\[\delta_{k-1}(H_0) \geq \left(\frac{1}{3} + \frac{3\eps}{4} \right)N_0.\] 
Choose $M\in \mathbb N$ such that $\mu, 1/\Delta \gg 1/M \gg 1/n$.
By \Cref{lemma:fewdiamonds-pairbalancer} with $t = t^\star$ (using the `moreover' part) there are vertices $a \in A$, $b\in B$, and an $(\{a\}, \{b\}, t^\star)$-gadget $L$ in $H_0$, with at most $M$ vertices, and such that $L$ contains exactly one $t^\star$-balancer.
    Let $J$ be the only $t^\star$-balancer of $L$, with anchor vertices $\{y_1, \dotsc, y_{t^\star}\}$ and $\{x_1, x_2\}$; and let $J_{i} \subseteq J$ be the half of $J$ containing $x_i$ for $i\in [2]$. 
    For each $i \in [2]$, let $L_i \subseteq L$ be isomorphic subgraphs as in the definition of $(\{x_1\}, \{x_2\}, t^\star)$-gadget, such that $J_i \subseteq L_i$.
    
 We define a partial embedding $\psi_1: \bigcup_{i \in [t^\star]} e_i \rightarrow H$  by setting $\psi_1(u) = x_1$ and $\psi_1(c_i) = y_i$ for all $i \in [t^\star]$, and arbitrarily mapping the vertices in $e_i \setminus \{u, c_i\}$ to the edge of $J_1$ containing $y_i$.
    Note that $J_1$ is immersed in the image of $\psi_1$.
    Finally, let $\mathcal{A}_0$ be the family of all half-diamonds in $L_1$, and
    $U_0 = \{ x_2 \}$. 
    \medskip 

\noindent \emph{Step 3: Finding parity-absorbers and balancers.}
    Recall that $\nu n \leq |V(T_0)| \leq \Delta \nu n$.
    By the pigeonhole principle, there must exist $t \in [\Delta]$ such that there are at least $\nu n/\Delta$ vertices of degree $t$ in $T_0$.
    Let $t$ be fixed.

    By \cref{Parity-absorbers}, there exists $\pi: V(H) \rightarrow \{0,1\}$ such that for every $w \in V(H)$, and every $(k-1)$-set $S \subseteq V(H)$ with $|S \cap A| \equiv \pi (w) \bmod 2$, there exists $m \leq M$ such that $|\mathcal{A}^{(t)}_m(w, S)| \geq \rho N^m$. Since $\rho\gg \delta'$, we can assume that $V(\mathcal{A}^{(t)}_m(w, S))\cap R=\emptyset $ for each $w$ and $S$.

    Again, by a probabilistic argument as in \Cref{lem:counting-absorb}, we can select a suitable family of pairwise-disjoint absorbers.
    We spell out the details for completeness.
    Let $1 \leq m \leq M$ be fixed.
    Select a family $\mathcal{A}'_{1,m}$ from the union of $\mathcal{A}^{(t)}_m(w, S)$ for all $w \in V(H)$ and $(k-1)$-sets $S \subseteq V(H)$ at random independently by including each $m$-set with probability $ \rho_1^2 N^{1-m}$.
    For a vertex $w \in V(H)$ and a $(k-1)$-set $ S \subseteq V(H)$, if $|\mathcal{A}^{(t)}_m(w, S)| \geq \rho_1 N^m$ then $\mathbb{E}[|\mathcal{A}^{(t)}_m(w, S)\cap \mathcal A'_{1,m}|]\ge\rho_1^3 N $.
   Moreover, $\mathbb{E}[|\mathcal{A}'_{1,m}|]\le \rho_1^2 N^{1-m}N^m = \rho_1^2 N$ and $\mathbb{E}[|(A_1, A_2): A_1 \text{~and~} A_2\ \text{in}\ \mathcal{A}'_{1,m} \text{ are intersecting}|]\le N^{2m-1}(\rho_1^2N^{1-m})^{2}=\rho_1^4m^2N$.
For each intersecting pair in $\mathcal{A}'_{1,m}$, remove one of its participating tuples.
By Chernoff’s inequality, Markov's inequality and a union bound, we 
    can get a subset $\mathcal A_{1,m}\subseteq \mathcal A'_{1,m}$ of size at most $\rho_1^3 N$ consisting of vertex-disjoint parity-absorbers, and $|\mathcal{A}^{(t)}_m(w, S) \cap\mathcal{A}_{1,m}|\geq \rho_1^4 N$,
    for each $w\in V(H)$ and $(k-1)$-set $S \subseteq V(H)$ with $|\mathcal{A}^{(t)}_m(w, S)| \geq \rho_1 N^m$.
We do this for each $1 \leq m \leq M$, obtaining vertex-disjoint families each time, and take the union $\mathcal{A}_1$ of all the obtained families $\mathcal{A}_{1,m}$.
The family $\mathcal{A}_1$ has at most $M \rho_1^2 N$ tuples, and for every $w\in V(H)$ and $(k-1)$-set $S \subseteq V(H)$ there exist some $1 \leq m \leq M$ such that 
\[
|\mathcal{A}^{(t)}_m(w, S) \cap\mathcal{A}_{1}|\geq \rho_1^4 N.
\] 
We can also assume that each element of $\mathcal A_1$ is vertex-disjoint from $\mathcal{A}_0$ and $U_0$.

Now, say that a $t$-gadget $L \subseteq H_0$ is an \emph{even-balancer} if there exist $Y_1 \subseteq A$, $Y_2 \subseteq B$ of size $2$, such that $L$ is a $(Y_1, Y_2, t)$-gadget.
Applying \Cref{lemma:fewdiamonds-pairbalancer} to $H_0$,  there is a family $\mathcal{A}_2$ of at least $2\rho_2 N$ pairwise vertex-disjoint even-balancers, disjoint from $\mathcal{A}_0 \cup \mathcal{A}_1 \cup U_0$.
   For each even-balancer $L$, let $L_1, L_2$ be isomorphic $k$-graphs such that $L$ is the edge-disjoint union of $L_1$ and $L_2$.
    By assumption, for each $L \in \mathcal{A}_2$ we have that there are $Y_1 \subseteq A$, $Y_2 \subseteq B$ of size $2$ each, such that $Y_i = V(L) \setminus V(L_{3-i})$, for $i \in [2]$.
    Now partition $\mathcal{A}_2$ into two parts $\mathcal{A}_2^1$ and $\mathcal{A}_2^2$ of size at least $\rho_2 N$ each.
    Let $\mathcal{A}'_2$ consist of $L_1$ for each $L \in \mathcal{A}_2^1$, and let $\mathcal{A}''_2$ consist of $L_2$ for each $L \in \mathcal{A}_2^2$. 
    Let $\mathcal{A} := \mathcal{A}_0 \cup \mathcal{A}_1 \cup \mathcal{A}_2'\cup \mathcal{A}_2''$.
    
    Apply \Cref{lemma:immersion} to embed $T^{(k)}_0$ with a partial embedding $\psi_2$ of $T^{(k)}$, in such a way that $\mathcal{A}$ is immersed in $\psi_2(T^{(k)}_0)$.
    Let $U_1$ be the set of uncovered vertices in the balancers of $\mathcal{A}_2$
    (namely, $U_1$ contains $Y_2$ for all $L \in \mathcal{A}^1_2$, and $Y_1$ for $L \in \mathcal{A}^2_2$).
    \medskip 

    \noindent \emph{Step 4: Embedding an almost spanning tree.}
    Let $T'(c_i)=T_{\mathrm{main}}^{(k)}\cap T^{(k)}(c_i)$ for $i\in[t^\star].$  
    Let $H_1:= H_0 - V(\psi_1(\bigcup_{i \in [t^\star]} e_i)) \cup V(\psi_2(T_0^{(k)}))  \cup U_0 \cup U_1$ and $|V(H_1)| := N_1$. By (\ref{equa-3}), we have $\delta_{k-1}(H_1) \geq (1/3 + \eps/2) N_1$.
    If $i \in [t^\star]$ is such that $|V(T'(c_i))| \leq \eps N_1/(2\Delta)$, then we can extend the embedding $\psi_1$ greedily to embed $T'(c_i)$ in $H_1$. Indeed, there are at most $\Delta-1$ many $i \in [t^\star]$ such that $|V(T'(c_i))| \leq \eps N_1/(2\Delta)$, so the resulting subgraph $H_2$ of $H_1$ has codegree at least $(1/3 + \eps/(2\Delta))|V(H_2)|$. Let $t^\star_1$ denote the number of subtrees satisfying $|V(T'(c_i))| > \eps N_1/(2\Delta)$ and let $|V(H_2)|:= N_2$.
    In the following, we partition the vertices of $H_2$ into $t_1^\star$ vertex sets, where each set has a size approximately equal to the size of the subtree $T'(c_i)$ for $i \in [t_1^\star]$.

    Note that $\psi_0(c_i)=y_i$ and $\deg_{H_2} (y_i) \geq (1/3 + \eps/(2\Delta))N_2$ for each $i\in[t_1^\star]$.
    We randomly partition $V(H_2)$ into $t_1^\star$ subsets $V_i$ satisfying $|V_i| = |V(T'(c_i))| + \delta N_2/(2\Delta)$.
    Using concentration inequalities and a union bound, we find a partition $\mathcal P=\{V_1, \dots, V_{t_1^{\star}}\}$ such that $\delta_{k-1}(H_2[V_i]) \geq (1/3 + \eps/(4\Delta))|V_i|$; and moreover $\deg_{H_2[V_i]} (y_i) \geq (1/3 + \eps/(4\Delta))|V_i|$ for each $i \in [t^\ast_1]$.
    Using \cref{almost-cover}, we can embed $T'(c_i)$ to $H[V_i]$ for each $i \in [t^\star_1]$ and ensure that $c_i$ is embedded to $y_i$. So we finish the embedding of $T^{(k)}_{\mathrm{main}}$ and denote the embedding by $\psi$.
    The embedding $\psi$ is obtained by extending the embedding $\psi_0$.
    
    Now, we choose $k-2$ vertices in the reservoir $R$ to connect $\psi_1(v)\in V(\psi_1(T_0^{(k)}))$ and $\psi(x)\in V(\psi(T^{(k)}_{\mathrm{main}}))$. Note that \[\deg_{H}(\{\psi_1(v), \psi(x)\}, R) \geq \delta' (\delta'N)^{k-2}/2.\] Then we find an edge to join the embedding of $T^{(k)}_{\mathrm{main}}$ to the embedding of $T_0^{(k)}$. Denote the total embedding by $\phi$.
    Let $U_2 \subseteq V(H) \setminus U_0$ be the set of uncovered vertices after this embedding.
    Note that $U_1\subseteq U_2$ and $|U_0 \cup U_2|$ must be divisible by $k-1$.
    \medskip
    
 \noindent \emph{Step 5: Balance the leftover.}
   For any even-balancer $L $ in  $\mathcal{A}_2^1$, we can switch two vertices in the leftover from $A$ to $B$  
  by changing the image of $ \psi_2^{-1}(L_1)$ from $L_1$ to $ L_2$. Similarly,  we can switch two vertices in the leftover from $B$ to $A$ using any even-balancer in  $\mathcal{A}_2^2$.
  In order to balance the leftover $U_2$, we use even-balancers in $\mathcal{A}_2$ iteratively. If $|U_2 \cap A| - |U_2 \cap B| >2$, then we use an even-balancer in  $\mathcal{A}_2^1$. If $|U_2 \cap B| - |U_2 \cap A| >2$, then we use an even-balancer in  $\mathcal{A}_2^2$.
  We stop the iteration when $|U_2 \cap A| = |U_2 \cap B| \pm 2$.
  Abusing notation, we let $\phi$ be the embedding of the $k$-expansion of $T_{\mathrm{main}} \cup T_0$ after applying those switchings.

    Let $U$ be the set of uncovered vertices after this process, and let $U' = U \setminus U_0$.
    Note that $|U|$ is divisible by $k-1$, and by construction we have that
    \[ |U' \cap A| = |U' \cap B| \pm 2, \]
    and \[ |U' \cup U_0| = 1 + |U_2|  \leq 1 + \delta n + 2 \rho_2 n \leq \alpha n/2.\]
    
    \noindent \emph{Step 6: Using the absorbers.}
    To finish the embedding we will extend $\phi$ by adding one $k$-edge at a time.
    Let $r$ be such that $r(k-1) = |U|$.
    Let $T'_0, T'_1, \dotsc, T'_r$ be a sequence of $k$-expansion hypertrees such that $T'_0$ is the $k$-expansion of $T_{\mathrm{main}} \cup T_0$ and $T'_r = T^{(k)}$, and for each $i \in [r]$, $T'_i$ is obtained from $T'_{i-1}$ by adding a leaf edge, which we call $f_i$.

    Let $c(k) \in \{0,1\}$ be such that $c(k) \equiv (k-1) \bmod 2$.
    We will inductively build a family of partial embeddings $\phi_0, \dotsc, \phi_{r}$ and decreasing families of absorbers $\mathcal{A}^*_0 \supseteq \mathcal{A}^*_1 \supseteq \dotsb \supseteq \mathcal{A}^*_r$ such that, for each $i \in \{0\} \cup [r]$, 
    \stepcounter{propcounter}
    \begin{enumerate}[label = {{\rm (\Alph{propcounter}\arabic{enumi})}}]
        \item \label{item:ABabsorption-phii} $\phi_i$ is an embedding of $T'_i$,
        \item \label{item:ABabsorption-balanced} for $U'_i := U' \setminus \phi_i(V(T'_i))$, we have $|U'_i \cap A| = |U'_i \cap B| \pm (2 + c(k))$,
        \item \label{item:ABabsorption-absorbers} each $(k-1)$-set has at least $ \rho_1^4 N - i$ absorbers in $\mathcal{A}^*_i$, and all of them are immersed by $\phi_i(T'_i)$,
        \item \label{item:ABabsorption-L0} if $i < r$, $U_0 \cap \phi_i(V(T'_i)) = \emptyset$.
    \end{enumerate}
    For $i = 0$ we take $\phi_0 = \phi$ and $\mathcal{A}^*_0 = \mathcal{A}_1$.
    Assuming $i \in [r-1]$ and that we have constructed $\phi_{i-1}$ and $\mathcal{A}^*_{i-1}$, we describe the construction of $\phi_{i}$ and $\mathcal{A}^*_{i}$ as follows.
    Recall that $U'_{i-1} = U' \setminus \phi_{i-1}(V(T'_{i-1}))$ is the current leftover that we can use.
    Note that, since $i < r$, we have that $|U'_{i-1}| = (r-i+1)(k-1)-1 \geq 2(k-1)-1 = 2k-3$ vertices.
    This bound, and \ref{item:ABabsorption-balanced}, imply that \[\min\{|U'_{i-1} \cap A|, |U'_{i-1} \cap B|\} \geq k-2 - c(k), \quad \max \{|U'_{i-1} \cap A|, |U'_{i-1} \cap B|\} \geq k-1.\]
    
    Let $v_i$ be the unique vertex in $V(T'_{i-1}) \cap f_i$ and $w_i = \phi_{i-1}(v_i)$.
    Now, we choose a $(k-1)$-set $Z \subseteq U'_{i-1}$ such that $|Z \cap A| \equiv \pi(w_i) \bmod 2$ and $|(U'_{i-1}\setminus Z)\cap A| = |(U'_{i-1}\setminus Z) \cap B| \pm 2+c(k)$, as follows.
    The idea is always to keep the leftover as balanced as possible in the vertices of $A$ and $B$; selecting more vertices from $A$ if there are more available vertices in $A$, and vice versa.
    More precisely, if $|U'_{i-1} \cap A| \leq |U'_{i-1} \cap B|$, let $z_a \leq (k - 1)/2$ be the maximal integer such that $z_a \equiv \pi(w_i) \bmod 2$; and let $Z_A \subseteq U'_{i-1} \cap A$ and $Z_B \subseteq U'_{i-1} \cap B$ be any sets of size $z_a$ and $k-1 - z_a$, respectively, and let $Z = Z_A \cup Z_B$.
    Note that the choice of $z_a$ implies $z_a \leq (k - 1 - c(k))/2$.
    Using $k \geq 3$ we can see this choice of $Z_A, Z_B$ is possible, since $z_a \leq (k-1 - c(k))/2 \leq k-2 - c(k) \leq |U'_{i-1} \cap A|$ and $k-1 - z_a \leq k-1 \leq |U'_{i-1} \cap B|$.
    Note that  $|U'_{i-1}\cap B|-|U'_{i-1}\cap A| \le 2+c(k)$ and $ |Z_B|-|Z_A|\le 2+c(k)$. So  $|(U'_{i-1}\setminus Z)\cap B| -|(U'_{i-1}\setminus Z) \cap A| \le2+c(k)$. 
    If $|U'_{i-1} \cap A| > |U'_{i-1} \cap B|$, a similar choice works, now choosing $z_a \geq (k-1)/2$ to be minimal such that $z_a \equiv \pi(w_i) \bmod 2$.
    This ensures that $z_a \geq (k-1+c(k))/2$, so $z_a \leq k-1 \leq |U'_{i-1} \cap A|$ and $k-1 - z_a \leq (k-1-c(k))/2 \leq k-2 \leq |U'_{i-1} \cap B|$.
 Thus  $|(U'_{i-1}\setminus Z)\cap A| -|(U'_{i-1}\setminus Z) \cap B| \le2+c(k)$. Observe that, in any case, this choice satisfies $|(U'_{i-1}\setminus Z)\cap A| = |(U'_{i-1}\setminus Z) \cap B| \pm 2+c(k)$. 

Next, since $|Z \cap A| \equiv \pi(w_i) \bmod 2$, there must exist a parity-absorber for $(w_i, Z)$ in $\mathcal{A}^*_{i-1}$.
Let $S\in N_H(w_i)$ and $L$ be a $t$-absorbing  $(S, Z, t)$-gadget for $(w_i, Z)$ consisting of two edge-disjoint isomorphic $k$-graphs $I^{L}_1$ and $I^{L}_2$ such that $S\subseteq V(I^{L}_2)$
and $I^{L}_2\in \mathcal{A}^*_{i-1}$ is a parity-absorber for $(w_i, Z)$.
Now we define $\phi_{i}$ from $\phi_{i-1}$ by setting $\phi_{i}(\phi_{i-1}^{-1}(I^{L}_2))=I^{L}_1$ and $\phi_{i}(f_i)=S\cup\{w_i\}$.
By construction, $\phi_{i}$ is an embedding of $\smash{T'_{i}}$ satisfying \ref{item:ABabsorption-phii} and \ref{item:ABabsorption-L0}.
For $U'_i := U' \setminus \phi_i(V(T'_i))=U'_{i-1}\setminus Z$, we have $|U'_i \cap A| = |U'_i \cap B| \pm 2 + c(k)$.
Let $\mathcal{A}^*_i:=\mathcal{A}^*_{i-1}\setminus\{I^{L}_2\}$.
So each $(k-1)$-set has at least $ \rho_1^4 N  - i$ absorbers in $\mathcal{A}^*_i$, and all of them are immersed by $\phi_i(T'_i)$.
    \medskip
   
    \noindent \emph{Step 7: The last absorbing step.}
    Finally, for $i=r$, assume that we have $\phi_{r-1}$ and $\mathcal{A}^*_{r-1}$ satisfying \ref{item:ABabsorption-phii}--\ref{item:ABabsorption-L0} for $i=r-1$.
    In this case, the leftover $U'_{r-1} \cup U_0$ has exactly $k-1$ vertices.
    Let $v_r$ be the unique vertex in $V(T'_{r-1}) \cap f_r$, and let $w_r = \phi_{i-1}(v_r)$.
    If $|(U'_{r-1} \cup U_0) \cap A|  \equiv \pi(w_r) \bmod 2$, then there is a parity-absorber for $(w_r, U'_r)$ and we are done.
    Otherwise, we can perform a switch using the $(\{x_1\}, \{x_2\}, t^\star)$-gadget $L$, this changes the embedding of $T'_{r-1}$ to liberate $x_1$ and to cover $U_0 = \{x_2\}$.
    Then the new leftover $U'' := (U'_{r-1} \setminus U_0) \cup \{x_1\}$ must satisfy $|U'' \cap A| \equiv \pi(w_r) \bmod 2$, so there is a parity-absorber for $(w_r, U'')$, which we can use to conclude.
\end{proof}

\section{Short proofs} \label{section:short}
\begin{proof}[Proof of \Cref{corollary:almost}]
    If $|V(T^{(k)})| < \eps n$, we can embed $T^{(k)}$ greedily by adding one edge at a time.
    Thus, we can assume that $|V(T^{(k)})| \geq \eps n$.
    Let $U \subseteq V(H)$ be a random subset of $|V(T^{(k)})| + k-1$ vertices.
    By our assumptions, we have that $\eps n \leq |U| \leq n$.
    Hence, standard concentration inequalities ensure that, with non-zero probability, $H[U]$ has minimum codegree at least $(1/3 + \eps/2)|U|$, and we fix such a set $U$ from now on.
    Now, let $v \in V(T)$ be any odd-degree vertex (e.g., a leaf).
    Let $T'$ be obtained from $T$ by adding a leaf to $v$.
    Then $T'$ has an even-degree vertex, and the $k$-expansion of $T'$ has exactly $|U|$ vertices.
    Then \cref{theorem:main-even} implies that $H[U]$ contains the $k$-expansion of $T'$, which in turn implies (since $T \subseteq T'$) that $H$ contains $T^{(k)}$.
\end{proof}

\begin{proof}[Proof of \Cref{theorem:nlogn}]
    Let $1/k, \eps \gg c \gg 1/n$.
    Let $T$ be an $n$-vertex tree with $\Delta_1(T) \leq c n / \log n$.
    Let $N = n + (k-2)(n-1)$ and let $H$ be an $N$-vertex $k$-graph with $\delta_{k-1}(H) \geq (1/2 + \eps)N$.
    It suffices to show that $T^{(k)} \subseteq H$.
    Let $U \subseteq V(H)$ be a random subset of $n$ vertices.
    For any $(k-1)$-set $X \subseteq V(H)$, we have that $\deg_H(X, U)$ is a random variable with expectation at least $(1/2 + \eps/2)n$.
    Standard probabilistic arguments imply that we can assume $U$ is such that for every $(k-1)$-set $X$, $\deg_H(X, U) \geq (1/2 + \eps/3)n$.

    Let $\{ Y_1, \dotsc, Y_{n-1} \}$ be a partition of $V(H) \setminus U$ into $n-1$ sets of size $k-2$ each.
    For $i \in [n-1]$, let $G_i$ be a graph on $U$ such that $u_1 u_2 \in E(G_i)$ if and only if $\{u_1, u_2\} \cup Y_i \in E(H)$.
    By our assumption on $U$, we have that $\delta(G_i) \geq (1/2 + \eps/3)n$ for each $i \in [n-1]$.
    Hence there is a `transversal' copy of $T$ in $\bigcup_{i \in [n-1]} G_i$, i.e., a copy of $T$ that uses exactly one edge from each $G_i$, by \cite[Theorem 1.1]{Montgomery-Muyesser-Pehova-2022}.
    By replacing edge $u_1 u_2$ in $G_i$ with the edge $\{u_1, u_2\} \cup Y_i \in E(H)$, we obtain a $k$-expansion of $T$ in $H$.
\end{proof}

\section{Concluding remarks} \label{section:conclusion}

\subsection{Loose trees}
It would be interesting to see if our techniques allow for the study of loose trees beyond $k$-expansion hypertrees.
Recall from the introduction that loose hypertrees can contain perfect matchings, so the thresholds in each case have to be above the ones ensuring the existence of perfect matchings.
Pehova and Petrova also conjectured that, in fact, a $d$-degree condition ensuring perfect matchings is also enough to embed all spanning bounded-degree loose trees~\cite[Conjecture 6.1]{Pehova-Petrova-24}.
The conjecture is true in the codegree case, i.e., $d = k-1$, by the main result of~\cite{Pavez-Matias-Nicolas-Stein-24}.
Recently, Chen and Lo~\cite{Chen-Lo-2025} obtained a set of technical conditions that ensure the existence of all bounded-degree spanning loose trees; in particular, their work verifies the conjecture of Pehova and Petrova for the case $d = k-2$.

\subsection{Loose Hamilton cycle thresholds} \label{section:conclusion-loosecycle}
Let us discuss further the connection between $\deltaLHC{k}{d}$ and the analogous threshold $\smash{\delta^{\mathrm{LHP}}_{k,d}}$ for the existence of loose Hamilton paths.
Those two thresholds coincide: this has not been proven formally in the literature but follows from known techniques.
One of the main insights from~\cite{Lang-Sanhueza-2024, Lang-Sanhueza-2022} is that tight Hamilton cycles can be embedded whenever the host hypergraph satisfies three key properties: \begin{enumerate*}
    \item it has perfect fractional matchings,
    \item it is suitably connected, and
    \item has `odd cycles' (or, equivalently, we have `lattice-completeness', i.e., no parity obstructions).
\end{enumerate*}
Moreover, those properties must be satisfied `robustly'.
The `robustness' is satisfied when the hypergraph is above the minimum degree thresholds.

From this point of view, and the fact that known lower bounds on $\deltaLHC{k}{d},\smash{\delta^{\mathrm{LHP}}_{k,d}}$ readily give (loose) connectivity, and the fact that paths and cycles are essentially equivalent for the verification of fractional matchings, the proof boils down to the verification of lattice-completeness.
In fact, Alvarado, Kohayakawa, Lang, Mota and Stagni~\cite[Theorem 9.3]{Alvarado-Kohayakawa-Lang-Mota-Stagni-2023} formally showed that $\deltaLHC{k}{d} \geq \smash{\delta^{\mathrm{LHP}}_{k,d}}$. Further details will appear in upcoming work \cite{Lang-Sanhueza-2025}, where the results of~\cite{Lang-Sanhueza-2024,Lang-Sanhueza-2022} for tight cycles are recovered for other types of cycles.

An interesting problem is to decide if $\deltaLHC{k}{d} \leq 1/2$ holds for all $k > d \geq 1$.
Again, trying to use the aforementioned framework, we can get connectivity and lattice-completeness easily in this setting.
Hence, a solution to the problem boils down to a question of `space' (i.e., perfect fractional matchings).
Note that a positive answer to this question, together with \Cref{theorem:main}, would imply the closed formula $\deltaEST{k}{d} = 1/2$ for all $k >d \geq 1$.

\subsection{Expansions of trees with at least one even-degree vertex}
It would be interesting to generalize \Cref{theorem:main-even} beyond codegree conditions, i.e., finding optimal minimum $d$-degree conditions to embed spanning $k$-expansion hypertrees whenever the original tree has at least one even-degree vertex.
In general, $\delta_d(H) \geq (1/3 + o(1))\binom{n-d}{k-d}$ is necessary, as shown by \Cref{theorem:construct-q} with $q=3$. Another lower bound is given by the asymptotic loose Hamilton cycle thresholds, so in any case we need $\delta_d(H) \geq (\max\{1/3, \deltaLHC{k}{d}\} + o(1))\binom{n-d}{k-d}$.
We believe this lower bound should reflect the truth.
In the smallest open case, we have $k=3$ and $d=1$, where $\deltaLHC{3}{1} = 7/16 > 1/3$ holds.

\subsection*{Acknowledgments}
The second author is extremely thankful to Guanghui Wang and his group at Shandong University for their extraordinary hospitality.
M. Rao was supported by the China Postdoctoral Science Foundation (2023M742091).
N. Sanhueza-Matamala was supported by ANID-FONDECYT Iniciaci\'{o}n Nº11220269 grant and ANID-FONDECYT Regular Nº1251121 grant.
G. Wang was supported by the Natural Science Foundation of China (12231018) and the Young Taishan Scholars Program of Shandong Province (201909001).
W. Zhou was supported by the Natural Science Foundation of China  (12401457), the China Postdoctoral Science Foundation (2024M761780),  the Natural Science Foundation of Shandong Province (ZR2024QA067) and Young Talent of Lifting engineering for Science and Technology in Shandong, China (SDAST2025QTA074).

\sloppy\printbibliography

\end{document}